\documentclass[11pt]{article}
\usepackage{multicol}
\usepackage{amsfonts}
\usepackage{latexsym,amssymb,amsmath,graphics,cite}

\usepackage[utf8]{inputenc}
\DeclareUnicodeCharacter{0394}{\ensuremath{\Delta}}
 \usepackage{xcolor}
 \usepackage{listings}
\usepackage{fancyhdr}
\usepackage{lipsum}
\usepackage{lineno}
\usepackage{tikz}
\usepackage{CJK}
\usepackage{graphicx}
\usepackage{multirow}
\usepackage{float}
\usepackage{caption}
\usepackage{diagbox}
\usepackage{multirow}
\usepackage{setspace}
\usepackage{appendix}
\usepackage{tabu}
\begin{document}
\newcommand{\qed}{\hphantom{.}\hfill $\Box$\medbreak}
\newcommand{\Proof}{\noindent{\bf Proof \ }}
\newcommand{\rnote}[1]{{\color{red}{\sf #1}}}

\newtheorem{theorem}{Theorem}[section]
\newtheorem{lemma}[theorem]{Lemma}
\newtheorem{proposition}[theorem]{Proposition}
\newtheorem{corollary}[theorem]{Corollary}
\newtheorem{remark}[theorem]{Remark}
\newtheorem{example}[theorem]{Example}
\newtheorem{definition}[theorem]{Definition}
\newtheorem{Construction}[theorem]{Construction}

\title{Geometric difference families  and  related \\ variable-weight geometric orthogonal codes\footnote{Supported by the Science Research Project of Hebei Education Department (Grant No. JCZX2026032), the Natural Science Foundation of Hebei Province (Grant No. A2025205023, A2023205045); the Graduate Innovation Project of School of Mathematical Sciences of Hebei Normal University (Grant No. ycxzzbs202602).}}

\author{Huizhen Gao$^{a}$, Yichen Bai$^{a}$, Xiaowei Su$^{b}$, Zihong Tian$^{a,c}$\thanks{Corresponding author. E-mail address: tianzh68@163.com.} \\
\small{$^{a}$School of Mathematical Sciences, Hebei Normal University, Shijiazhuang 050024,   China}\\
\small {$^{b}$ College of Basic Education, Beijing Institute of Economics and Management, Beijing 100102,   China}\\
\small{$^{c}$Hebei Research Center of the Basic Discipline Pure Mathematics, Shijiazhuang 050024,   China}}

\date{}

\maketitle

\noindent {\bf Abstract:}
 Variable-weight geometric orthogonal codes (GOCs) were introduced  by Doty and Winslow for designing the macrobonds in 3D DNA origami. Recently,   Wang et al. introduced  the concept of geometric difference families (GDFs) and established the equivalence between GOCs and a certain type of GDFs.  Based on the relationship, Su et al. gave  two classes of  variable-weight  GOCs. In this paper, we first prove the existence of  $(n\times m, K, 1)$-GDFs for $K=\{3,5\}$ and  $\{3,6\}$. For further research,  we completely determine the existence of all the remaining parameters for an $(n\times m, \{3,4,5\}, 1)$-GDF  by Su et al.    As consequences, the  perfect $(n\times m,K,1)$-GOCs   for $K= \{3,5\}$, $\{3,6\}$ and $\{3,4,5\}$  are obtained.\vspace{0.2cm}


\section{Introduction}
Let $N $ and $M$ be two sets of integers containing $0$ such that if $a\in N$ (or $a\in M$) then $-a\in N$ (or $-a\in M$). Let  $K$ be a set of positive integers. For any $k$-subset $B$ of $N \times M$, where $k\in K$, define the multi-set
$$\Delta B :=\{(x_{1}-x_{2},y_{1}-y_{2}): (x_{1},y_{1}), (x_{2},y_{2})\in B, (x_{1},y_{1})\neq (x_{2},y_{2})\}$$
as the \emph{list of differences} from $B$. An \emph{$(N \times M,K,1)$-geometric difference packing} (briefly GDP) $\mathcal{B}$ is a collection of subsets (called \emph{base~blocks}) of $N \times M$ of sizes from $K$ such that the multi-set $\Delta\mathcal{B}:=\bigcup_{B\in\mathcal{B}}\Delta B$ covers every element of $N \times M$ at most once. $(N \times M)\setminus \Delta\mathcal{B}$ is called the \emph{difference leave} or \emph{leave} of this packing. If the leave consists only of $(0,0)$, $\mathcal{B}$ is called an \emph{$(N \times M,K,1)$-geometric difference family} (briefly GDF). As usual, an $(N \times M,\{k\},1)$-GDP (or -GDF) will be abbreviated to an $(N \times M,k,1)$-GDP (or -GDF).

For convenience, in what follows we always assume that $[a,b]=\{a,a+1,\dots,b\}$ for integers $a,b$ with $a<b$, and
$$[n]:=[-\frac{n-1}{2},\frac{n-1}{2}]$$
for any odd integer $n$, unless otherwise specified.
We usually write an $(n\times m,K,1)$-GDP (or -GDF) instead of an $([n]\times [m],K,1)$-GDP (or -GDF). Obviously, an $(n\times m,K,1)$-GDP (or -GDF) implies an $(m\times n,K,1)$-GDP (or -GDF).

When $m=1$, an $(n\times m,K,1)$-GDP (or -GDF) is simply written as an $(n,K,1)$-GDP (or -GDF), and the second coordinates of all elements in the base blocks will be omitted.  Bermond et al. \cite{Bermond} proved that geometric difference families do not exist for any $k\geq6$. For $k=3,4$, the existence problems have been completely settled.
However, the existence results for $k=5$ are even more scarce. Here, we list some known results for $(n,k,1)$-GDFs.

\begin{lemma}{\rm (Abel et al. \cite{Difference-families}, Beth et al. \cite{Beth}, Ge et al. \cite{GeG}, Liu et al. \cite{Liu})}\label{n,3,1}\\
$(1)$~There exists an $(n,3,1)$-GDF if and only if $n\equiv1,7\pmod{24};$\\
$(2)$~There exists an $(n,4,1)$-GDF if and only if $n \equiv 1\pmod{12}$, $n \geq 13$, and $n \notin \{25,37\};$\\
$(3)$~There exists a $(20t+1,5,1)$-GDF for $t=6,8,10;$\\
$(4)$~There are no $(n,5,1)$-GDFs for $n\equiv21\pmod{40}$~or~$n=41,81.$
\end{lemma}

$(n,K,1)$-GDFs are interesting not only in their own right, but also in their relation to many other combinatorial configurations such as semi-cyclic holey group divisible designs \cite{Feng} and optimal variable-weight optical orthogonal codes \cite{Jiang, WuD, yang}. We quote some results for later use. Note that an $(n,\{k,s^{*}\},1)$-GDF is an $(n,\{k,s\},1)$-GDF with exactly one base block of size $s$, and the number of base blocks of size $k$ is greater than zero.

\begin{lemma}{\rm (Wu et al. \cite{WuD})}\label{n-345}\\
$(1)$~There exists an $(n,\{3,4^{*}\},1)$-GDF if and only if $n\equiv1\pmod{6}$ and $n\geq19;$\\
$(2)$~There exists an $(n,\{3,5^{*}\},1)$-GDF if and only if $n\equiv9,15\pmod{24}$ and $n\geq33;$\\
$(3)$~There exists an $(n,\{3,6^{*}\},1)$-GDF if and only if $n\equiv1\pmod{6}$ and $n\geq43.$
\end{lemma}

When $m>1$,  $(n\times m,k,1)$-GDFs were first introduced by Wang et al. \cite{WLD}, which was motivated by the application in constructions of geometric orthogonal codes. They gave a complete solution to the existence of an $(n\times m,3,1)$-GDF. Since  $(n\times m,K,1)$-GDFs can be used to construct variable-weight geometric orthogonal codes, Su et al. \cite{Su} continue  to discuss the existence of $(n\times m,K,1)$-GDFs with $|K|\geq2$. They proved the existence of an $(n\times m, \{3,4\}, 1)$-GDF and an $(n\times m, \{3,4,5\}, 1)$-GDF except possibly for a few values.

\begin{lemma}{\rm(Wang et al. \cite{WLD})}\label{14}
There exists an $(n\times m, 3,1)$-GDF if and only if $nm\equiv 1\pmod{6}$ and $n,m\equiv 1,7,17,23 \pmod{24}$.
\end{lemma}

\begin{lemma}{\rm (Su et al. \cite{Su})}\label{result34}
There exists an $(n\times m, \{3,4\},1)$-GDF if and only if $nm\equiv 1\pmod{6}$.
\end{lemma}

\begin{lemma}{\rm (Su et al. \cite{Su})}\label{result345}
There exists an $(n\times m, \{3,4,5\}, 1)$-GDF if and only if $n, m\equiv1\pmod{2}, n,m \neq 3$ and $\{n,m\}\notin\{\{1,p\},\{5,7\},\{5,9\}\}$, where $p\in \{5,9,11,15,17,21,23,$ $27,29,35,41,47,53\}$, and possibly except for $\{n,m\}\in\{\{5,13\},\{5,45\},\{7,23\},\{7,29\},\{7,35\},$ $ \{9,35\},\{11,19\},\{11,27\}, $  $\{13,17\},\{13,21\},\{13,23\},\{13,27\},\{13,29\},\{13,35\},\{15,17\},\{15, $ $21\},\{15, 27\},\{17,21\},\{17, 27\},$ $\{21, 21\},\{21, 27\},\{23, 45\},\{27, 27\},\{29, 45\},\{35,45\}\}$.
\end{lemma}

In this paper, we mainly study  the existence of $(n\times m,K,1)$-GDFs  for $K=\{3,5\}$ and  $\{3,6\}$. First, we    consider the case of $m=1$, i.e., the existence of $(n, K,1)$-GDFs. Combining the necessary conditions for the existence of  $(n\times m,K,1)$-GDFs with several recursive constructions, we classify the parameters $n$ and $m$ and reduce the original existence problem to that of a minimal set of small-parameter designs. These small-parameter designs are obtained via computer search, and their existence is central to this work. As a further work, we  establish the existence of  the remaining parameters for $K=\{3,4,5\}$ in Lemma \ref{result345}. The main results of the paper are the following theorems.

\begin{theorem}\label{result35}
There  exists an $(n\times m, \{3,5\}, 1)$-GDF if and only if $n, m\equiv1,7\pmod{8}$, $\{n,m\}\neq\{1,p\}$, where~$p\in \{9,15,17,23,41,47\}$.
\end{theorem}

\begin{theorem}\label{result36}
There  exists an $(n\times m, \{3,6\}, 1)$-GDF if and only if $nm\equiv1\pmod{6}$, $\{n,m\}\notin \{\{1,p\},\{5,5\}\}$, where~$p\in \{13,19,37\}$.
\end{theorem}

\begin{theorem}\label{xinresult345}
There  exists an $(n\times m, \{3,4,5\}, 1)$-GDF if and only if $n, m\equiv1\pmod{2},$ $ n,m\neq 3$ and $\{n,m\}\notin\{\{1,p\},\{5,7\},\{5,9\},\{5,13\}\}$, where $p\in \{5,9,11,15,17,21,23,27,29,$ $35,41,47,53\}$.
\end{theorem}

The rest of the paper is organized as follows.  Section 2  introduces  some auxiliary designs.   Section 3  gives several recursive constructions for  geometric difference packings. The existence of an $(n\times m,\{3,5\},1)$-GDF and an $(n\times m,\{3,6\},1)$-GDF is completely proved in Section 4 and Section 5 respectively. Section 6  presents the existence of all the remaining parameters of an $(n\times m,\{3,4,5\},1)$-GDF. In Section 7, we consider the application of  geometric difference families to variable-weight geometric orthogonal codes. Conclusions and discussions are given in Section 8.

\section{Auxiliary designs}
In this section, we introduce some auxiliary designs, which will play important roles in the
recursive constructions for  geometric difference families.

\subsection{Semi-perfect group divisible designs}
Let $K$ be a set of positive integers. A \emph{group divisible design}, denoted by $K$-GDD, is a triple
$(X, \mathcal{G}, \mathcal{B})$ where $X$ is a finite set of  points, $\mathcal{G}$ is a partition of $X$ into  subsets (called \emph{groups}),   $\mathcal{B}$ is a collection of subsets (called \emph{blocks}) of $X$ such that  each
block has size from  $K$ and every 2-subset of $X$ occurs either in some group or in exactly one block, but not both. The
multi-set $T = \{|G| : G \in \mathcal{G}\}$ is called the \emph{type} of the $K$-GDD. If $\mathcal{G}$ contains $u_i$ groups of size $g_i$ for $1\leq i\leq r$, then we also denote the type by $g_1^{u_1} g_2^{u_2}\dots g_r^{u_r}$.  If $K=\{k\}$, we write a $\{k\}$-GDD as a $k$-GDD.

Let $S$ be a set of $n$ points, $m$ an odd integer and $X=S\times[m]$, $\mathcal{G}=\{\{i\}\times [m]$: $i\in S\}$.
Let $\mathcal{F}$ be a collection of subsets (called \emph{base blocks}) of $S\times[m]$. For $i,j\in S$ and $B\in{\cal F}$, define a multi-set $\Delta_{ij}(B)=\{x-y: (i,x),(j,y)\in B, (i,x)\neq(j,y)\}$, and a multi-set $\Delta_{ij}({\cal F})=\bigcup_{B\in{\cal F}}\Delta_{ij}(B)$. If for any $i,j\in S$,
$$\Delta_{ij}({\cal F})=\left\{
\begin{array}{lll}
[m], & \ \ i\neq j;\\
\emptyset, & \ \ i=j,\\
\end{array}
\right.
$$
then a $K$-GDD of type $m^n$ with the point set $X'=S\times [0,m-1]$ together with the group set ${\cal {G}}' =\{\{i\} \times [0,m-1]: i\in S\}$ can be generated from ${\cal F}$, where $K=\{|B|:B\in{\cal F}\}$. The required blocks are
obtained by developing all base blocks of ${\cal F}$ by successively
adding $1$ to the second component of each point of these base
blocks modulo $m$. We call $\mathcal{F}$ a \emph{semi-perfect group divisible design}, denoted by $K$-SPGDD of type $m^n$.
As usual, a $\{k\}$-SPGDD will be abbreviated to a $k$-SPGDD.

\begin{example}\label{9SPGDD}
All base blocks of a $3$-SPGDD of type $9^5$ on $[0,4]\times [9]$ are listed below.\vspace{0.2cm}

\hspace{0.5cm}$\begin{array}{lll}
\{(0, 0), (1, 4), (2, 0)\},~~~ & \{(0, 1), (1, 4), (3, 0)\},~~~ & \{(0, 2), (1, 4), (4, 0)\},\\
\{(0, 3), (1, 3), (4, 0)\}, & \{(0, 0), (2, 4), (4, 0)\}, & \{(0, 0), (1, 1), (3, 4)\},\\
\{(0, 0), (2, 1), (4, 4)\}, & \{(0, 4), (1, 3), (2, 0)\}, & \{(0, 0), (2, 2), (3, 3)\},\\
\{(0, 0), (2, 3), (4, 3)\}, & \{(1, 0), (2,4), (3, 0)\}, & \{(1, 3), (2, 1), (3, 0)\},\\
\end{array}$

\hspace{0.5cm}$\begin{array}{lll}\{(0, 4), (1, 0), (2, 3)\}, & \{(1, 1), (2, 3), (4, 0)\}, & \{(1, 2), (2, 3), (3, 0)\},\\
\{(0, 4), (1, 1), (3, 0)\}, & \{(0, 2), (3, 0), (4, 3)\}, & \{(0, 4), (3, 1), (4, 0)\},\\
\{(1, 0), (3, 1), (4, 3)\}, ~~~& \{(0, 2), (1, 0), (4, 4)\},~~~ & \{(2, 2), (3, 0), (4, 1)\},\\
\{(0, 1), (3, 3), (4, 0)\}, & \{(0, 3), (2, 0), (3, 4)\}, & \{(0, 2), (2, 0), (3, 2)\},\\
\{(1, 2), (2, 2), (4, 0)\}, & \{(1, 0), (3, 4), (4, 0)\}, & \{(2, 0), (3, 3), (4, 1)\},\\
\{(1, 1), (2, 0), (4, 2)\}, & \{(1, 0), (3, 2), (4, 2)\}, & \{(2, 0), (3, 0), (4, 4)\}.\\
\end{array}$
\end{example}

\begin{example}\label{15SPGDD}
All base blocks of a $3$-SPGDD of type $15^5$ on $[0,4]\times [15]$ are listed below.\vspace{0.2cm}

\hspace{0.5cm}$\begin{array}{lll}
\{(0, 0), (1, 7), (2, 0)\},~~~ & \{(0, 1), (1, 7), (3, 0)\},~~~ & \{(0, 2), (1, 7), (4, 0)\},\\
\{(0, 3), (1, 6), (4, 0)\},~~~ & \{(0, 0), (2, 7), (4, 0)\},~~~ & \{(0, 2), (2, 7), (3, 0)\},\\
\{(0, 4), (1, 6), (3, 0)\},~~~ & \{(0, 2), (1, 6), (2, 0)\},~~~ & \{(0, 0), (1, 0), (3, 7)\},\\
\{(0, 0), (1, 1), (4, 7)\},~~~ & \{(0, 0), (2, 1), (3, 6)\},~~~ & \{(0, 0), (2, 2), (4, 5)\},\\
\{(0, 0), (2, 4), (4, 6)\},~~~ & \{(0, 0), (2, 3), (3, 5)\},~~~ & \{(0, 0), (2, 6), (3, 4)\},\\
\{(1, 0), (2, 7), (3, 1)\},~~~ & \{(0, 7), (1, 0), (2, 6)\},~~~ & \{(1, 1), (2, 6), (4, 0)\},\\
\{(1, 2), (2, 5), (4, 0)\},~~~ & \{(1, 1), (2, 5), (3, 0)\},~~~ & \{(0, 0), (3, 0), (4, 4)\},\\
\{(0, 1), (3, 4), (4, 0)\},~~~ & \{(0, 7), (2, 4), (4, 0)\},~~~ & \{(1, 2), (2, 4), (3, 0)\},\\
\{(1, 3), (2, 3), (4, 0)\},~~~ & \{(0, 4), (3, 5), (4, 0)\},~~~ & \{(1, 0), (3, 6), (4, 0)\},\\
\{(0, 4), (1, 0), (4, 5)\},~~~ & \{(0, 0), (3, 2), (4, 3)\},~~~ & \{(1, 4), (3, 7), (4, 0)\},\\
\{(0, 6), (1, 0), (2, 1)\},~~~ & \{(0, 5), (1, 0), (4, 7)\},~~~ & \{(1, 0), (3, 5), (4, 2)\},\\
\end{array}$

\hspace{0.5cm}$\begin{array}{lll}
\{(1, 4), (3, 0), (4, 5)\},~~~ & \{(1, 3), (3, 0), (4, 7)\},~~~ & \{(2, 0), (3, 0), (4, 6)\},\\
\{(1, 4), (2, 0), (4, 7)\},~~~ & \{(1, 2), (2, 0), (3, 6)\},~~~ & \{(0, 7), (1, 5), (3, 0)\},\\
\{(2, 0), (3, 7), (4, 5)\},~~~ & \{(1, 1), (2, 0), (3, 3)\},~~~ & \{(0, 3), (1, 0), (3, 0)\},\\
\{(2, 0), (3, 4), (4, 4)\},~~~ & \{(0, 7), (2, 0), (3, 1)\},~~~ & \{(2, 3), (3, 0), (4, 3)\},\\
\{(0, 6), (1, 5), (2, 0)\},~~~ & \{(1, 5), (2, 2), (4, 0)\},~~~ & \{(0, 6), (3, 1), (4, 0)\},\\
\{(0, 5), (2, 1), (4, 0)\},~~~ & \{(2, 1), (3, 0), (4, 2)\}.\\
\end{array}$
\end{example}

SPGDDs are closely related to  perfect difference matrices  (PDMs).  A $k\times m$ matrix $D=(d_{ij})$ with entries from $[m]$ is called a {\em perfect difference matrix}, denoted by  PDM$(k,m)$, if for all $1\leq s<t\leq k$, the list of differences $\{d_{sj}-d_{tj}:1\leq j\leq m\}$ contains each element of $[m]$ exactly once. A $k$-SPGDD  of type $m^k$ is equivalent to a PDM$(k,m)$.

\begin{lemma}{\rm(Ge et al. \cite{GeG}, Wang et al. \cite{WLD})}\label{345PDM}\\
$(1)$~There exists a PDM$(3,m)$ for any odd integer $m\geq3;$\\
$(2)$~There exists a PDM$(4,m)$~for any odd integer $3<m<200$ expect for $m=9,11$ and possibly for $m=59;$\\
$(3)$~There exists a PDM$(5,m)$ for any $m=5^{a_1}121^{a_2}161^{a_3}201^{a_4}$, where $a_1,a_2,a_3,a_4$ are non-negative integers, not all equal to $0$.
\end{lemma}

\subsection{Semi-perfect modified group divisible designs}

A {\em modified group divisible design} (MGDD) is a quadruple $(X,{\cal H},{\cal G},{\cal B})$ where $X$ is a finite set of $mn$   points, ${\cal G}$ is a partition of $X$ into $n$ subsets (called {\em groups}) with  size $m$, ${\cal H}$ is another partition of $X$ into $m$ subsets (called {\em holes}) with size $n$ satisfying that  $|H\cap G|=1$ for any $H\in{\cal H}$ and $G\in{\cal G}$, $\cal B$ is a set of subsets (called {\em blocks}) of $X$ such that each block has size from a set of
positive integers $K$ and no block contains two distinct points of any group or any hole, but any other pair of distinct points of $X$ occurs in exactly one block of $\cal B$.
Such a design is denoted by a $K$-MGDD of type $m^n$.   A $\{k\}$-MGDD  is usually abbreviated as  a $k$-MGDD.

\begin{lemma}{\rm (Abel et al. \cite{MGDD}, Cao et al. \cite{caoh})}\label{MGDD}\\
$(1)$ There exists a $3$-MGDD of type $m^n$ if and only if $n, m\geq3$, $(n-1)(m-1)\equiv0\pmod {2}$ and $n(n-1)m(m-1)\equiv0\pmod {6};$\\
$(2)$ There exists a $4$-MGDD of type $m^n$ if and only if $n, m\geq4$ and $(n-1)(m-1)\equiv0\pmod {3}$ except for $(n,m)=(4,6);$\\
$(3)$ There exists a $5$-MGDD of type $5^5$.
\end{lemma}

Let $S$ be a set of $n$ points and $X=S\times[m]$, $\mathcal{G}=\{\{i\}\times [m]:$ $i\in S\}$ and $\mathcal{H}=\{S\times\{x\}:x\in[m]\}$.
Let $\mathcal{F}$ be a collection of subsets (called \emph{base blocks}) of $S\times[m]$. For $i,j\in S$ and $B\in{\cal F}$, define  multi-sets $\Delta_{ij}(B)$  and   $\Delta_{ij}({\cal F})$ as in Section 2.1. If for any $i,j\in S$,
$$\Delta_{ij}({\cal F})=\left\{
\begin{array}{lll}
[m]\setminus\{0\}, & \ \ i\neq j;\\
\emptyset, & \ \ i=j,\\
\end{array}
\right.
$$
then a $K$-MGDD of type $m^n$ with the point set $X'=S\times [0,m-1]$ together with the group set ${\cal {G}}' =\{\{i\} \times [0,m-1]: i\in S\}$ and the hole set ${\cal H}'=\{S\times \{x\}: x\in[0,m-1]\}$ can be generated from ${\cal F}$, where $K=\{|B|:B\in{\cal F}\}$. The required blocks are
obtained by developing all base blocks of ${\cal F}$ by successively
adding $1$ to the second component of each point of these base
blocks modulo $m$. We call $\mathcal{F}$ a \emph{semi-perfect modified group divisible design}, denoted by $K$-SPMGDD of type $m^n$.
Clearly, if $\mathcal{F}$ is a $K$-SPMGDD of type $m^n$ on $S\times [m]$, then $\mathcal{F}\cup\{S\times\{0\}\}$ forms a $(K\cup\{n\})$-SPGDD of type $m^n$.

\section{Construction methods}
In this section, we  present some recursive constructions for  geometric difference packings.

Let $H_1$ and $H_2$ be two sets of integers. For any positive integers $r_1$ and $r_2$, define
$$H_1^{r_1}\times H_2^{r_2}:=\{(r_1 x_1, r_2 x_2):(x_1,x_2)\in H_1\times H_2\}.$$
When $H_i=[h_i]$ for some odd positive integer $h_i$, we write $H_i^{r_i}$ as $[h_i]^{r_i}$. When $r_i=1$, we omit $r_i$ from the notation. An $(N \times M,K,1)$-GDF can be viewed as an $(N \times M,K,1)$-GDP with leave $[1]^{r_1}\times [1]^{r_2}$ for any positive integers $r_1$ and $r_2$. Further, if there exists an $(N \times M,K,1)$-GDP with leave $H_1^{r_1}\times H_2^{r_2}$, then $H_i$ contains $0$ and if $a\in H_i$, then $-a\in H_i$, where $i=1,2$.

\begin{lemma}\label{origin}{\rm (Su et al. \cite{Su}, Wang et al. \cite{WLD})}
There exists an $(m,3,1)$-GDP with leave $[h]$ for any $(m,h)$ satisfying one of the following conditions$:$\vspace{0.1cm}

$(1)$~$m\equiv1,7\pmod{24}$,  $m\geq49$ and  $h=7,$ or  $m\geq175$ and  $h=25;$

$(2)$~$m\equiv3,21\pmod{24}$,  $m\geq21$ and $h=3,$  or $m\geq147$ and $h=21;$

$(3)$~$m\equiv5,11\pmod{24}$,  $m\geq35$ and $h=5,$  or $m\geq77$ and $h=11;$

$(4)$~$m\equiv9,15\pmod{24}$, $m\geq63$ and $h=9,$  or $m\geq105$ and $h=15;$

$(5)$~$m\equiv13,19\pmod{24}$,  $m\geq91$ and $h=13,$  or $m\geq133$ and $h=19;$

$(6)$~$m\equiv17,23\pmod{24}$,  $m\geq119$ and $h=17,$  or $m\geq161$ and $h=23$.
\end{lemma}

\begin{Construction}{\rm (Wang et al. \cite{WLD})}\label{nm-nxm,3,1}
If there exists an $(n,3,1)$-GDP with leave ${[h_1]}^{r_1}$ and an $(m,3,1)$-GDP with leave ${[h_2]}^{r_2}$,  then there exists an $(n\times m,3,1)$-GDP with leave ${[h_1]}^{r_1}\times {[h_2]}^{r_2}$.
\end{Construction}

\begin{Construction}\label{budong}{\rm (Su  et al. \cite{Su})}
Suppose that there exists\vspace{0.1cm}

\noindent $(1)$ an $(n\times m,K,1)$-GDP with leave $H_1^{r_1}\times H_2^{r_2}$, and

\noindent $(2)$ an $(H_1\times H_2,L,1)$-GDP with leave $T_1^{s_1}\times T_2^{s_2}$,\vspace{0.1cm}

\noindent then there exists an $(n\times m,K\cup L,1)$-GDP with leave $T_1^{r_1s_1}\times T_2^{r_2s_2}$.
Moreover, if the  $(H_1\times H_2,L,1)$-GDP is a GDF, then the resulting GDP is also a GDF.
\end{Construction}

\begin{Construction}\label{N-NxM-GDP}{\rm (Su  et al. \cite{Su})}
Suppose that there exists an $(n,K,1)$-GDP with leave $H^r.$ If
there exists an $L$-SPGDD of type $m^k$ for each $k\in K$, then
there exists an $(n\times m, L,1)$-GDP with leave $H^r\times [m]$.
\end{Construction}

\begin{Construction}\label{MGDD-SPMGDD}{\rm (Su  et al. \cite{Su})}
Suppose there exists an $(m,K,1)$-GDF. If there exists an $L$-MGDD of type $k^h$ for each $k\in K$, then there exists an $L$-SPMGDD of type $m^h$.
\end{Construction}

Langford sequence is an important combinatorial structure,  which has been used to construct  geometric difference packing by Chee et al. \cite{cklh}. A {\em Langford sequence of order $n$ and defect $d$}, $n>d$, is a partition of $[1,2n]$ into a collection of ordered pairs $(a_i,b_i)$ such that $\{b_i-a_i : 1 \leq i \leq n\}=[d,d+n-1]$.

\begin{lemma}{\rm (Simpson \cite{LD})}\label{LDjieguo}
A Langford sequence of order $n$ and defect $d$  exists if and only if
~$(1)$~$n\geq2d-1$, and
~$(2)$~$n\equiv0,1\pmod{4}$ for odd $d$, or $n\equiv0,3\pmod{4}$ for even $d$.
\end{lemma}

\begin{Construction}{\rm (Chee et al. \cite{cklh})}\label{LD-GDP}
If there exists a  Langford sequence of order $n$ and defect $d$, then there exists a $(6n+2d-1,3,1)$-GDP with leave $[2d-1]$.
\end{Construction}

\begin{Construction}{\rm (Wu et al. \cite{WuD})}\label{6m+u}
Suppose that there exists a $(u,K,1)$-GDF and a Langford sequence of order $m$ and defect $\frac{u+1}{2}$, then there exists a $(6m+u,K\cup\{3\},1)$-GDF.
\end{Construction}

\section{ $(n\times m, \{3,5\},1)$-GDFs}
In this section, we investigate the existence of $(n\times m,\{3,5\},1)$-GDFs. We first consider the case of $m=1$, then some $(n\times m, \{3,5\},1)$-GDFs can be obtained from $(n, \{3,5\},1)$-GDFs by applying the constructions given in Section $3$.  With the help of computer search, some GDFs with small parameters are also obtained. In addition, an $(n\times m,3,1)$-GDF or an $(n\times m,5,1)$-GDF is obviously also an $(n\times m,\{3,5\},1)$-GDF.

\begin{lemma}\label{non}
There is no $(n\times m,\{3,5\},1)$-GDF for $n\equiv3,5\pmod{8}$ and $m\equiv1\pmod{2}$.
\end{lemma}
\Proof Suppose that $\mathcal{B}$ is an $(n\times m,\{3,5\},1)$-GDF. Let the base blocks  be $\{(a_{i_1},b_{i_1}),$ $(a_{i_2},b_{i_2}),(a_{i_3},b_{i_3})\}$~and~$\{(a_{j_1},b_{j_1}),(a_{j_2},b_{j_2}),(a_{j_3},b_{j_3}),(a_{j_4},b_{j_4}),(a_{j_5},b_{j_5})\}$, where $1\leq i\leq x_3,$ $1\leq j\leq x_5$, whose elements are ordered lexicographically. Then the sum of the first coordinates of those differences lying in the two right quadrants in $\Delta\mathcal{B}$ is
$\sum_{i=1}^{x_3}2(a_{i_3}-a_{i_1})+\sum_{j=1}^{x_5}2(2a_{j_5}+a_{j_4}-a_{j_2}-2a_{j_1})$,  which is even.
On the other hand,  the sum of the first coordinates of those differences lying in the two right quadrants in $\Delta\mathcal{B}$ is $(1+2+3+\cdots + \frac{n-1}{2})m=\frac{(n-1)(n+1)m}{8}$ by the property of geometric difference families.

Let $n=2k+1$, then $\frac{(n-1)(n+1)m}{8}=\frac{2k(2k+2)m}{8}=\frac{k(k+1)m}{2}.$
When  $k=4t,$   $\frac{k(k+1)}{2}=2t(4t+1)$ is even;
When  $k=4t+1,$   $\frac{k(k+1)}{2}=(2t+1)(4t+1)$ is odd;
When $k=4t+2,$   $\frac{k(k+1)}{2}=(2t+1)(4t+3)$ is odd;
When  $k=4t+3,$  $\frac{k(k+1)}{2}=(2t+2)(4t+3)$ is even. Thus, when $n\equiv3,5\pmod{8}$ and $m\equiv1\pmod{2}$,
$\frac{(n-1)(n+1)m}{8}$ is odd,   a contradiction. Therefore, there is no $(n\times m,\{3,5\},1)$-GDF for  $n\equiv3,5\pmod{8}$~and~$m\equiv1\pmod{2}$. \qed

\begin{lemma}\label{yiwei}
There is an $(n,\{3,5\},1)$-GDF if and only if $n\equiv1,7\pmod{8}$ and $n\notin\{9, 15, 17,$ $23, 41, 47\}$.
\end{lemma}
{\Proof} Suppose $\mathcal{B}$ is an $(n,\{3,5\},1)$-GDF on $[n]$ and the number of base blocks of sizes $3,5$ in $\mathcal{B}$ is $x,y$, respectively. Considering the list of differences for  base blocks in $\mathcal{B}$, we have
\begin{equation}\label{nece}
6x+20y=n-1,
\end{equation}
then $n\equiv 1\pmod{2}$. There is no $(n,\{3,5\},1)$-GDF for $n\in\{9,15\}$ as Equation (\ref{nece}) has no solution. According to Lemma~\ref{non}, there is no $(n,\{3,5\},1)$-GDF for $n\equiv3,5\pmod{8}$. Further,  there exists an $(n,\{3,5\},1)$-GDFs  for $n\equiv1,7\pmod{24}$ or an $n\equiv9,15\pmod{24}$ and $n\geq33$ by Lemma  \ref{n,3,1} (1) and Lemma  \ref{n-345} (2).

 Next, we consider the case of $n\equiv17,23\pmod{24}$. There is a $(65,\{3,5\},1)$-GDF by Appendix A. By Lemma \ref{LDjieguo},  there is a {\em Langford sequence} of order $m$ and defect $33$ for $m\geq65$ and $m\equiv0,1\pmod{4}$. Specifically speaking, when $m\equiv0\pmod{4}$  and $m\geq68$,   $6m+65\equiv17\pmod{24}$ and $6m+65\geq473;$ when $m\equiv1\pmod{4}$ and $m\geq65$,  $6m+65\equiv23\pmod{24}$ and $6m+65\geq455$. Therefore, there is an $(n,\{3,5\},1)$-GDF  for $n\equiv17\pmod{24}$ and $n\geq473$ or $n\equiv23\pmod{24}$ and $n\geq455$ by Construction \ref{6m+u}.

Consider the case of $n\equiv17\pmod{24}$ and $n<473$. There is no $(17,\{3,5\},1)$-GDF as Equation (\ref{nece}) has no solution. When $n=41$, Equation (\ref{nece}) has a solution $(x,y)=(0,2)$, but there is no $(41,5,1)$-GDF by Lemma \ref{n,3,1} (4). So there is no $(41,\{3,5\},1)$-GDF. A  $(161,\{3,5\},1)$-GDF exists by Lemma \ref{n,3,1} (3). For $n\in\{65,89,113,137,185,209,233,257,281,$ $305, 329, 353,377,401,425,449\}$, see Appendix A. Further, we consider the case of $n\equiv23\pmod{24}$ and $n<455$. There is no $(23,\{3,5\},1)$-GDF as Equation (\ref{nece}) has no solution. There is no $(47,\{3,4,5\},1)$-GDF by Lemma \ref{result345}, then there is no $(47,\{3,5\},1)$-GDF. For $n\in\{71,95,119,143,167,191,215,239,263,287,311,335,$ $359,383,407, 431\}$, see Appendix A.

This completes the proof.
\qed

\begin{lemma}\label{75SPGDD}
If there exists an $(m,\{3,5\},1)$-GDF$,$ then there exists a $\{3,5\}$-SPGDD of type $m^5$.
\end{lemma}
{\Proof} Suppose there is an $(m,\{3,5\},1)$-GDF. By Lemma \ref{MGDD} (1) (3), there is a $\{3,5\}$-MGDD of type~$3^5$ and a $\{3,5\}$-MGDD of type~$5^5$.  Applying  Construction \ref{MGDD-SPMGDD}, we get a $\{3,5\}$-SPMGDD of type $m^5$ on $A\times [m]$ with the set of base blocks  $\mathcal{F}$, where $A$ is a set of 5 points.
  Further, set~$\mathcal{C}=\mathcal{F}\cup\{A\times\{0\}\}$, then~$\mathcal{C}$ forms a $\{3,5\}$-SPGDD of type $m^5$.
\qed

\begin{lemma}\label{yiweierweiGDF}
If there exists an $(n,\{3,5\},1)$-GDF and an $(m,\{3,5\},1)$-GDF, then there exists an $(n\times m,\{3,5\},1)$-GDF.
\end{lemma}

{\Proof} Suppose there is an $(n,\{3,5\},1)$-GDF and an $(m,\{3,5\},1)$-GDF. By Lemma \ref{75SPGDD}, there is a $\{3,5\}$-SPGDD of type $m^5$. By Lemma \ref{345PDM} (1), there is a PDM$(3,m)$, i.e., a 3-SPGDD of type $m^3$. Therefore, there exists an $(n\times m, \{3,5\},1)$-GDF from Constructions \ref{budong} and  \ref{N-NxM-GDP}.
\qed

In what follows, we discuss the existence of $(n\times m,\{3,5\},1)$-GDFs according to the different values of $n,m$ by applying recursive constructions given in Section 3. For each $(n,m)$, we  need to find the appropriate $(n,3,1)$-GDPs with leave $[h_1]^{r_1}$,   $(m,3,1)$-GDPs with leave $[h_2]^{r_2}$ and   $(h_1\times h_2,\{3,5\},1)$-GDFs.

By Lemma \ref{non}, the necessary condition for the existence of an $(n\times m,\{3,5\},1)$-GDF is  $n,m\equiv1,7\pmod{8}$. We classify $n,m$ according to the values modulo $24$.

\begin{lemma}\label{3-GDF}
When $n,m\equiv1,7\pmod{24}$~or~$n,m\equiv17,23\pmod{24}$, there exists an $(n\times m,\{3,5\},1)$-GDF.
\end{lemma}
{\Proof}  When~$n,m\equiv1,7\pmod{24}$~or~$n,m\equiv17,23\pmod{24}$, we have $nm\equiv1\pmod{6}$. By Lemma \ref{14}, there is an $(n\times m,3,1)$-GDF, so there exists an $(n\times m,\{3,5\},1)$-GDF.
\qed

\begin{lemma}
When $n\equiv 1,7\pmod{24}$, $m \equiv 9,15,17,23\pmod{24}$, there exists an $(n\times m, \{3,5\},1)$-GDF.
\end{lemma}
 \Proof By Lemmas \ref{n,3,1}, \ref{yiwei} and \ref{yiweierweiGDF}, there exists an $(n\times m,\{3,5\},1)$-GDF for $n\equiv 1,7\pmod{24}$ and  $m \equiv 9,15,17,23\pmod{24}, m\notin\{9,15,17,23,41,47\}$.

When $n\equiv 1,7\pmod{24}$ and $n\geq 49$, there is an $(n,3,1)$-GDP with leave $[7]$ by Lemma \ref{origin} (1). There is an $(n,3,1)$-GDP with leave $[7]^3$ for $n\in \{25,31\}$ by Appendix D and a $(7,3,1)$-GDP with leave $[7]$. When $m\in\{9,15,17,23,41,47\}$, there is a PDM$(3,m)$,
 i.e., a $3$-SPGDD of type $m^{3}$ by Lemma \ref{345PDM} (1). Applying Construction \ref{N-NxM-GDP}, we get an $(n\times m,3,1)$-GDP with leave $[7]^{r}\times [m]$, where $r=1$ or $3$.  When $m\in\{9,15,17,23\}$, there exists  a $(7\times m,\{3,5\},1)$-GDF by Appendix B. When $m\in\{41,47\}$,  there exists  an $(m,\{3,5\},1)$-GDP with leave $[9]^{3}$ by Appendix D. In addition, there is a PDM$(3,7),$ i.e., a $3$-SPGDD of type $7^{3}$ by Lemma \ref{345PDM} (1);
 there is a $\{3,5\}$-SPGDD of type $7^{5}$ by Lemma \ref{75SPGDD}. Applying Construction \ref{N-NxM-GDP}, we get a  $(7\times m,3,1)$-GDP with leave $[7]\times [9]^{3}$. Further, there is a $(7\times9, \{3,5\},1)$-GDF by Appendix B.  Applying Construction \ref{budong}, there exists  a $(7\times m,\{3,5\},1)$-GDF for $m\in\{41,47\}$.

From the above, the conclusion holds.
\qed

\begin{lemma}\label{zuihou35}
When $n\equiv 9,15\pmod{24}$, $m \equiv 9,15,17,23\pmod{24}$, there exists an $(n\times m, \{3,5\},1)$-GDF.
\end{lemma}
 \Proof  By Lemmas  \ref{yiwei} and \ref{yiweierweiGDF}, there exists an $(n\times m,\{3,5\},1)$-GDF for $n\equiv 9,15\pmod{24}, n\notin\{9,15\}$ and $m \equiv 9,15,17,23\pmod{24}, m\notin\{9,15,17,23,41,47\}$.

Next, we  consider the existence of $(n\times m, \{3,5\},1)$-GDFs for $n\equiv 9,15\pmod{24}$ and $m \in\{ 9,15,17,23,$ $41, 47\}$,  or $n\in\{9,15\}$ and $m \equiv 9,15,17,23\pmod{24}, m\notin\{9,15,17,23,41,47\}$.\vspace{0.1cm}

$(1)$ When $n\equiv 9,15\pmod{24}$ and $n\geq 63$, there is an $(n,3,1)$-GDP with leave $[9]$ by Lemma \ref{origin} (4).
There is an $(n,3,1)$-GDP with leave $[9]^3$ for $n\in \{33,39,57\}$ by Appendix D and a $(9,3,1)$-GDP with leave $[9]$.

 When $m\in\{9,15,17,23,41,47\}$, there is a PDM$(3,m),$
 i.e., a $3$-SPGDD of type $m^{3}$ by Lemma \ref{345PDM} (1). Applying Construction \ref{N-NxM-GDP}, we get an $(n\times m,3,1)$-GDP with leave
 $[9]^{r}\times [m]$,  where $r=1$ or $3$.  There exists a $(9\times m,\{3,5\},1)$-GDF for $m\in\{9,15,17,23\}$ by Appendix B.
  For $m\in\{41,47\}$, there exists an $(m,\{3,5\},1)$-GDP with leave $[9]^{3}$ by Appendix D. In addition,
 there is a PDM$(3,9)$, i.e., a $3$-SPGDD of type $9^{3}$ by Lemma \ref{345PDM} (1);
 there is a $3$-SPGDD of type $9^{5}$ by Example \ref{9SPGDD}. Applying Construction \ref{N-NxM-GDP}, we get a  $(9\times m,3,1)$-GDP with leave $[9]\times [9]^{3}$. A $(9\times9, \{3,5\},1)$-GDF exists by Appendix B. Applying Construction \ref{budong},  there exists a   $(9\times m, \{3,5\},1)$-GDF for $m\in\{41,47\}$.
Therefore, there exists an $(n\times m,\{3,5\},1)$-GDF for   $n\equiv 9,15\pmod{24}$, $n\neq 15$ and $m \in\{ 9,15,17,23,41, 47\}$ by applying Construction \ref{budong}.

When $n=15$, there is a  $(15,3,1)$-GDP with leave $[15]$. There exists a $(15\times m,\{3,5\},1)$-GDF for $m\in\{9,15,17,23\}$ by Appendix B.  For $m\in\{41,47\}$, there exists an $(m,\{3,5\},1)$-GDP with leave $[9]^{3}$ by Appendix D.
 There is a PDM$(3,15)$, i.e., a $3$-SPGDD of type $15^{3}$ by Lemma \ref{345PDM} (1); there is a $3$-SPGDD of type $15^{5}$ by Example \ref{15SPGDD}.  Applying Construction \ref{N-NxM-GDP}, we get a  $(15\times m,3,1)$-GDP with leave $[15]\times [9]^{3}$. A  $(9\times15, \{3,5\},1)$-GDF exists by Appendix B.  Applying Construction \ref{budong},  there exists a   $(15\times m, \{3,5\},1)$-GDF for $m\in\{41,47\}$.\vspace{0.1cm}

$(2)$ When $n\in\{9,15\}$, there is an $(n,3,1)$-GDP with leave $[n]$. When $m \equiv 9,15\pmod{24}$ and $m\geq 63$, there is an $(m,3,1)$-GDP with leave $[9]$ by Lemma \ref{origin} (4).  There is an $(m,3,1)$-GDP with leave $[9]^3$ for $m\in \{33,39,57\}$ by Appendix D.
There exists a $(9\times9, \{3,5\},1)$-GDF and a $(9\times15, \{3,5\},1)$-GDF by Appendix B. So there exists an $(n\times m,\{3,5\},1)$-GDF for   $n\in\{9,15\}$ and $m \equiv 9,15\pmod{24}, m\notin\{9,15\}$  by applying Constructions  \ref{nm-nxm,3,1} and \ref{budong}.

 When $m \equiv 17,23\pmod{24}$ and $m\geq 119$, there is an $(m,3,1)$-GDP with leave $[17]$ by Lemma \ref{origin} (6).  There is a  $(65,3,1)$-GDP with leave $[17]^2$,  an $(m,3,1)$-GDP with leave $[17]^3$ for  $m\in\{71,113\}$, an $(m,3,1)$-GDP with leave $[23]^{3}$ for
  $m\in\{89,95\}$ by Appendix D.   Further, there exists a
$(9\times17, \{3,5\},1)$-GDF, a $(15\times17, \{3,5\},1)$-GDF, a $(9\times23, \{3,5\},1)$-GDF and a $(15\times23, \{3,5\},1)$-GDF by Appendix B. Therefore, there exists an $(n\times m,\{3,5\},1)$-GDF for   $n\in\{9,15\}$ and $m \equiv 17,23\pmod{24}, m\notin\{17,23,41,47\}$  by applying Constructions  \ref{nm-nxm,3,1} and \ref{budong}.

This completes the proof.\qed\vspace{0.2cm}

\noindent {\bf \emph{Proof of Theorem} \ref{result35}}  Combining the results of Lemmas \ref{non}-\ref{yiwei}, \ref{3-GDF}-\ref{zuihou35},  the theorem is proved.\qed

\section{$(n\times m, \{3,6\}, 1)$-GDFs}
In this section, we study the existence   of $(n\times m, \{3,6\}, 1)$-GDFs. We first consider the case of $m=1$. Some $(n\times m, \{3,6\}, 1)$-GDFs can be obtained from $(n ,\{3,6\}, 1)$-GDFs by applying the constructions given in Section $3$. Note that an $(n\times m, 3, 1)$-GDF is obviously also an $(n\times m, \{3,6\}, 1)$-GDF.

\begin{lemma}\label{3.6}
There is an $(n, \{3,6\}, 1)$-GDF if and only if $n\equiv 1\pmod{6}$ and $n\notin \{13,19,37\}$.
\end{lemma}
{\Proof} Suppose  $\mathcal{B}$ is an $(n,\{3,6\},1)$-GDF on $[n]$, and the number of base blocks of sizes $3,6$ in $\mathcal{B}$ is $x,y$, respectively. Considering the list of differences of  base blocks in $\mathcal{B}$, we have
\begin{align}\label{36}
6x+30y=n-1,
\end{align}
then $n\equiv 1\pmod{6}$. By Lemma \ref{n-345} (3), there is an  $(n, \{3,6\}, 1)$-GDF for $n\equiv 1\pmod{6}$ and $n\geq 43$. Next, we consider the case of $n\in\{7,13,19,25,31,37\}$.

 When $n\in\{13,19\}$, from   Equation (\ref{36}), there are only the base blocks of size $3$, so there is no $(n,3,1)$-GDF by Lemma \ref{n,3,1} (1). When $n=37$, Equation ($\ref{36}$) has two solutions, $(x,y)=(6,0)$ and $(1,1)$. There is no $(37, 3, 1)$-GDF by Lemma \ref{n,3,1} (1) and there is no  $(37, \{3,6\}, 1)$-GDF by Lemma \ref{n-345} (3).
 When $n\in\{7,25,31\}$,   an $(n,3,1)$-GDF exists by Lemma \ref{n,3,1} (1), then an $(n, \{3,6\},1)$-GDF exists.
So, the conclusion holds. \qed

\begin{lemma}
If there exists an $(n\times m, \{3,6\}, 1)$-GDF, then $n,m\equiv 1,7,13,19\pmod{24}$ or $n,m\equiv 5,11,17,23\pmod{24}$.
\end{lemma}
{\Proof} Suppose   $\mathcal{B}$ is an $(n\times m, \{3,6\}, 1)$-GDF on $[n]\times [m]$ and the number of base blocks of sizes $3,6$ in $\mathcal{B}$ is $x,y$, respectively. Considering the list of differences of  base blocks in $\mathcal{B}$, we have
\begin{align}\label{erwei36}
6x+30y=nm-1,
\end{align}
then $nm\equiv 1\pmod{6}$. We classify $n$ and $m$ according to the values modulo $24$, then $n,m\equiv 1,7,13,19\pmod{24}$ or $n,m\equiv 5,11,17,23\pmod{24}$.
\qed

\begin{lemma}
When $n,m \equiv 1,7,13,19 \pmod{24}$, there exists an $(n\times m, \{3,6\}, 1)$-GDF.
\end{lemma}
{\Proof} $(1)$   $n,m\equiv 1,7 \pmod{24}$.

There exists an $(n\times m, \{3,6\}, 1)$-GDF, since  an $(n\times m,3,1)$-GDF exists by Lemma \ref{14}.\vspace{0.1cm}

$(2)$   $n\equiv 1,7 \pmod{24}$, $m\equiv 13,19 \pmod{24}$.

For $n\equiv1,7\pmod{24}$, when $n\geq 49$, there exists an $(n,3,1)$-GDP with leave $[7]$ by Lemma \ref{origin} (1); when  $n\in \{25,31\}$, there exists    an $(n,3,1)$-GDP with leave $[7]^{3}$ by Appendix D. There exists a  $(7,3,1)$-GDP with leave $[7]$.

When $m\equiv13,19\pmod{24}$ and $m\geq91$, there exists an $(m,3,1)$-GDP with leave $[13]$ by Lemma \ref{origin} (5). When  $m\in \{43,61,67,85\}$, there exists an $(m,3,1)$-GDP with leave $[13]^{3}$ by Appendix D, and there exists a  $(13,3,1)$-GDP with leave $[13]$. Applying Construction \ref{nm-nxm,3,1}, we can obtain an $(n\times m, 3, 1)$-GDP with leave $[7]\times [13]$. Further, there exists a $(7\times 13, \{3,6\}, 1)$-GDF by Appendix C. So there exists an $(n\times m, \{3,6\}, 1)$-GDF by  Construction \ref{budong}.

When $m\in\{19,37\}$, there exists a PDM$(3,m)$  by Lemma \ref{345PDM} (1), i.e.,  a  $3$-SPGDD of type $m^{3}$, and there exists a $(7\times m, \{3,6\}, 1)$-GDF by Appendix C.  Applying Construction \ref{N-NxM-GDP}, we  obtain an $(n\times m, 3, 1)$-GDP with leave $[7]\times [m]$. So   there exists an  $(n\times m, \{3,6\}, 1)$-GDF   by Construction  \ref{budong}.\vspace{0.1cm}

$(3)$   $n,m\equiv 13,19 \pmod{24}$.

When $n,m\equiv 13,19 \pmod{24}$ and $n,m\geq 91$, there exists an $(n,3,1)$-GDP with leave $[13]$ by Lemma \ref{origin} (5). There is a $(13\times 13, \{3,6\}, 1)$-GDF by Appendix C. So  there exists an  $(n\times m, \{3,6\}, 1)$-GDF by  Constructions \ref{nm-nxm,3,1} and \ref{budong}.

Consider the case of $n\equiv13,19\pmod{24}$ and $n\geq91$, $m\in\{13,19,37,43,61,67,85\}$. There exists an $(n,3,1)$-GDP with leave $[13]$ by Lemma \ref{origin} (5). When $m\in \{43,61,67,85\}$, there exists an $(m,3,1)$-GDP with leave $[13]^{3}$ by Appendix D, and  there exists a  $(13,3,1)$-GDP with leave $[13]$. Further, a $(13\times 13, \{3,6\}, 1)$-GDF exists by Appendix C. Then there exists an  $(n\times m, \{3,6\}, 1)$-GDF  by  Constructions \ref{nm-nxm,3,1} and \ref{budong}.
When $m\in\{19,37\}$, there exists   a $3$-SPGDD of type $m^{3}$ and a $(13\times m, \{3,6\}, 1)$-GDF by Appendix C. Then there exists an  $(n\times m, \{3,6\}, 1)$-GDF   by  Constructions  \ref{N-NxM-GDP} and \ref{budong}.

Consider the case of $n,m\in\{13,19,37,43,61,67,85\}$. If $n,m\in \{43,61,67,85\}$, there exists an $(n,3,1)$-GDP with leave $[13]^{3}$ by Appendix D and a  $(13,3,1)$-GDP with leave $[13]$. Further, a $(13\times 13, \{3,6\}, 1)$-GDF exists by Appendix C.  Then there exists an  $(n\times m, \{3,6\}, 1)$-GDF   by  Constructions \ref{nm-nxm,3,1} and \ref{budong}.
If $n\in\{19,37\}$, $m\in\{13,43,61,67,85\}$, there exists   a $3$-SPGDD of type $n^{3}$ and  a $(13\times n, \{3,6\}, 1)$-GDF   by Appendix C.  Then there exists an  $(n\times m, \{3,6\}, 1)$-GDF   by  Constructions  \ref{N-NxM-GDP} and \ref{budong}.
If $n,m\in\{19,37\}$,  there exists a $(19\times 19, \{3,6\}, 1)$-GDF, a $(19\times 37, \{3,6\}, 1)$-GDF and a $(37\times 37, \{3,6\}, 1)$-GDF by Appendix C.\qed

\begin{lemma}\label{555}
There is no $(5\times 5,\{3,6\},1)$-GDF.
\end{lemma}
{\Proof} If a $(5\times 5,\{3,6\},1)$-GDF exists, then it contains only the base blocks of size 3 and no base blocks of size 6 from Equation (\ref{erwei36}). By Lemma \ref{14}, there is no $(5\times 5,\{3,6\},1)$-GDF.
\qed

\begin{lemma}\label{zuihou36}
When $n,m \equiv 5,11,17,23\pmod{24}$, there exists an $(n\times m, \{3,6\}, 1)$-GDF.
\end{lemma}
{\Proof} $(1)$   $n,m\equiv 17,23 \pmod{24}$.

There exists an $(n\times m, \{3,6\}, 1)$-GDF,  since   an $(n\times m,3,1)$-GDF exists  by Lemma \ref{14}.

$(2)$   $n,m \equiv 5,11\pmod{24}$.

When $n,m\equiv5,11\pmod{24}$ and $n,m\geq77$, there exists an $(n,3,1)$-GDP with leave $[11]$ and an $(m,3,1)$-GDP with leave $[11]$ by Lemma \ref{origin} (3). There exists a  $(11\times 11, \{3,6\}, 1)$-GDF by Appendix C.  So  there exists an  $(n\times m, \{3,6\}, 1)$-GDF   by  Constructions \ref{nm-nxm,3,1} and \ref{budong}.

Consider the case of $n\equiv5,11\pmod{24}$ and $n\geq77$, $m\in\{5,11,29,35,53,59\}$. When $m\in\{29,35\}$, there exists an $(m,3,1)$-GDP with leave $[5]^{3}$ by Appendix D and  a  $(5,3,1)$-GDP with leave $[5]$. When $m\in\{11,53,59\}$, there exists an $(m,3,1)$-GDP with leave $[11]^{3}$ by Appendix D and  a $(11,3,1)$-GDP with leave $[11]$. Further, a $(11\times 5, \{3,6\}, 1)$-GDF and a $(11\times 11, \{3,6\}, 1)$-GDF exist by Appendix C. So  there exists an  $(n\times m, \{3,6\}, 1)$-GDF   by  Constructions \ref{nm-nxm,3,1} and \ref{budong}.

Consider the case of $n,m\in\{5,11,29,35,53,59\}$. There is no $(5\times5,\{3,6\},1)$-GDF by Lemma \ref{555}. If $n,m\in\{29,35\}$, there exists  an $(n,3,1)$-GDP with leave $[5]^{3}$  by Appendix D and a PDM$(3,m)$  by Lemma \ref{345PDM} (1), i.e., a $3$-SPGDD of type $m^{3}$. Then  there exists an  $(n\times m, \{3,6\}, 1)$-GDP  with leave $[5]^3\times [m]$ by  Construction \ref{N-NxM-GDP}.
If $n\in\{29,35\}$, $m\in\{11,53,59\}$, there exists an $(n,3,1)$-GDP with leave $[5]^{3}$, an $(m,3,1)$-GDP with leave $[11]^{3}$ by Appendix D and there exists a $(11,3,1)$-GDP with leave $[11]$. Then there exists an  $(n\times m, \{3,6\}, 1)$-GDP  with leave $[5]^3\times [11]^r$ by  Construction \ref{nm-nxm,3,1}, where $r=1$ or 3. If $n,m\in\{11,53,59\}$, then there exists an  $(n\times m, \{3,6\}, 1)$-GDP  with leave $[11]^r\times [11]^r$ by  Construction \ref{nm-nxm,3,1}, where $r=1$ or 3.  Further,
there exists a $(5\times 11, \{3,6\}, 1)$-GDF, a $(5\times 29, \{3,6\}, 1)$-GDF, a $(5\times 35, \{3,6\}, 1)$-GDF and a $(11\times 11, \{3,6\}, 1)$-GDF by Appendix C. So   there exists an  $(n\times m, \{3,6\}, 1)$-GDF   by  Construction  \ref{budong}.
\vspace{0.1cm}

$(3)$  $n\equiv 5,11 \pmod{24}$, $m\equiv 17,23 \pmod{24}$.

When $n\equiv5,11\pmod{24}$ and $n\geq35$, there exists an $(n,3,1)$-GDP with leave $[5]$ by Lemma \ref{origin} (3). When $m\equiv17,23\pmod{24}$ and $m\geq119$, there exists an $(m,3,1)$-GDP with leave $[17]$ by Lemma \ref{origin} (6). Applying Construction \ref{nm-nxm,3,1}, we get an $(n\times m, 3, 1)$-GDP with leave $[5]\times [17]$. From Appendix D, there exists an $(m,3,1)$-GDP with leave $[11]^{2}$ for $m\in \{41,47\}$, a  $(65,3,1)$-GDP with leave $[17]^{2}$,  an $(m,3,1)$-GDP with leave $[17]^{3}$ for $m\in \{71,113\}$, an $(m,3,1)$-GDP with leave $[23]^{3}$ for $m \in \{89,95\}$. And  there exists a  $(17,3,1)$-GDP with leave $[17]$,  a  $(23,3,1)$-GDP with leave $[23]$.
 Further, there exists a $(5\times 11, \{3,6\}, 1)$-GDF, a $(5\times 17, \{3,6\}, 1)$-GDF and  a $(5\times 23, \{3,6\}, 1)$-GDF
 by Appendix C.   So   there exists an  $(n\times m, \{3,6\}, 1)$-GDF   by  Constructions  \ref{nm-nxm,3,1} and  \ref{budong}.

Consider the case of $n\in\{5,11,29\}$, $m\equiv 17,23 \pmod{24}$ and $m\geq119$. There exists an $(m,3,1)$-GDP with leave $[17]$ by Lemma \ref{origin} (6). There exists a  $(5,3,1)$-GDP with leave $[5]$ and a $(29,3,1)$-GDP with leave $[5]^{3}$ by Appendix D.  Applying Construction \ref{nm-nxm,3,1}, we get an $(n\times m, 3, 1)$-GDP with leave $[5]^r\times [17]$ for $n\in\{5,29\}$, where $r=1$ or $3$. When $n=11$, there exists a PDM(3,11), i.e., a $3$-SPGDD of type $11^{3}$ by Lemma \ref{345PDM} (1). Applying Construction \ref{N-NxM-GDP}, there exists a $(11\times m, 3, 1)$-GDP with leave $[11]\times [17]$. Further, there exists a  $(17,3,1)$-GDP with leave $[5]^{2}$ by Appendix D and a $(5\times 17, \{3,6\}, 1)$-GDF by Appendix C.   So   there exists an  $(n\times m, \{3,6\}, 1)$-GDF   by  Constructions  \ref{nm-nxm,3,1} and  \ref{budong}.

Consider the case of $n\in\{5,11,29\}$, $m\in\{17,23,41,47,65,71,89,95,113\}$. There exists a  $(5,3,1)$-GDP with leave $[5]$, a $(11,3,1)$-GDP with leave $[11]$, a $(29,3,1)$-GDP with leave $[5]^{3}$, a $(23,3,1)$-GDP with leave $[5]^{2}$ and a $(17,3,1)$-GDP with leave $[5]^{2}$ by Appendix D. There exists a  $3$-SPGDD of type $11^{3}$ by Lemma \ref{345PDM} (1).
 Regarding the value of $m$, it is the same as the first case of this category. Further, there exists a   $(5\times 11, \{3,6\}, 1)$-GDF, a   $(5\times 17, \{3,6\}, 1)$-GDF, a   $(5\times 23, \{3,6\}, 1)$-GDF, a   $(11\times 11, \{3,6\}, 1)$-GDF,  a   $(11\times 17, \{3,6\}, 1)$-GDF by Appendix C.
 Applying Constructions \ref{nm-nxm,3,1}-\ref{N-NxM-GDP}, there exists an  $(n\times m, \{3,6\}, 1)$-GDF.
\qed

\noindent {\bf \emph{Proof of Theorem} \ref{result36}}   Combining the results of Lemmas \ref{3.6}-\ref{zuihou36},  the theorem is proved.\qed

\section{$(n\times m, \{3,4,5\}, 1)$-GDFs}
In this section, we discuss the existence of the remaining parameters for $(n\times m, \{3,4,5\}, 1)$-GDFs in Lemma \ref{result345}.
\begin{lemma}\label{non345}
There is no $(5\times 13, \{3,4,5\},1)$-GDF.
\end{lemma}
{\Proof} Suppose that there exists a  $(5\times 13, \{3,4,5\},1)$-GDF with base block set $\mathcal{B}$. Denote the set of base blocks of size $i$ in $\mathcal{B}$ by $\mathcal{B}_{i}$ and $|\mathcal{B}_{i}|= x_{i}$ for $i\in\{3,4,5\}$, respectively. For $ B\in \mathcal{B}_{i},$ $i\in\{3,4,5\}$, define
\begin{center}
  $\Delta B_{i,0} = \{(0,y): (0,y)\in \Delta B,  y>0\}$,
  $\Delta B_{i,j} = \{(j,y): (j,y)\in \Delta B\},  j=1,2$,
\end{center}
and
\begin{center}
    $\Delta \mathcal{F}_{j} = \bigcup\limits_{i\in\{3,4,5\}}\bigcup\limits_{B\in \mathcal{B}_i}(\Delta B_{i,j}), ~j=0,1,2$.
\end{center}
From the property of perfect difference families, we have $|\Delta \mathcal{F}_{0}|=6$, $|\Delta \mathcal{F}_{1}|=|\Delta \mathcal{F}_{2}|=13$. By listing all of the possible forms of the base blocks, we have
\begin{flushleft}
  $(|\Delta B_{3,0}|,|\Delta B_{3,1}|,|\Delta B_{3,2}|)$ = $(3,0,0)$ or $(1,2,0)$ or $(1,0,2)$ or $(0,2,1)$;\\

  $(|\Delta B_{4,0}|,|\Delta B_{4,1}|,|\Delta B_{4,2}|)$ = $(6,0,0)$ or $(3,3,0)$ or $(3,0,3)$ or $(2,4,0)$  or $(2,0,4)$ or $(1,3,2)$,\\
  ~~~~~~~~~~~~~~~~~~~~~~~~~~~~~~~~~~~~~~$(1,4,1)$;\\
  $(|\Delta B_{5,0}|,|\Delta B_{5,1}|,|\Delta B_{5,2}|)$ = $(10,0,0)$ or $(6,4,0)$ or $(6,0,4)$ or $(4,6,0)$  or $(4,0,6)$ or $(3,4,3)$,\\
  ~~~~~~~~~~~~~~~~~~~~~~~~~~~~~~~~~~~~~~$(3,6,1)$ or $(2,6,2)$ or $(2,4,4)$;\\
\end{flushleft}

 On the other hand, we have $|\Delta \mathcal{B}|= 6x_{3}+12x_{4}+20x_{5}=64$. It is straightforward to check that the equation has three solutions $(x_{3},x_{4},x_{5})=(4,0,2)$ or $(0,2,2)$ or $(2,1,2)$. $(x_{3},x_{4},x_{5})=(4,0,2)$  is impossible by Theorem \ref{result35}. When $(x_{3},x_{4},x_{5})=(0,2,2)$, since $|\Delta \mathcal{F}_{0}|=6$, then $|\Delta B_{4,0}|\leq 1$, $|\Delta B_{5,0}|\leq 2$. It contradicts   $|\Delta \mathcal{F}_{1}|=13$.
 When $(x_{3},x_{4},x_{5})=(2,1,2)$, since $|\Delta \mathcal{F}_{0}|=6$, $|\Delta \mathcal{F}_{1}|=|\Delta \mathcal{F}_{2}|=13$, after  detailed calculation,  there is only one possibility:  there are two base blocks of size 3 in the form of $(|\Delta B_{3,0}|,|\Delta B_{3,1}|,$ $|\Delta B_{3,2}|)=(1,0,2)$ and $(0,2,1)$, a base block of  size 4 in the form of $(|\Delta B_{4,0}|,|\Delta B_{4,1}|,|\Delta B_{4,2}|)=(1,3,2)$,  two base blocks of   size 5 in the form of $(|\Delta B_{5,0}|,|\Delta B_{5,1}|,$ $|\Delta B_{5,2}|)= (2,4,4)$. It can be determined that such a $(5\times 13, \{3,4,5\},1)$-GDF does not exist by exhaustive computer search, the program is listed in Appendix F.

Therefore,  there is no $(5\times 13, \{3,4,5\},1)$-GDF.
 \qed

\begin{lemma}\label{on345}
There is an $(n\times m,\{3,4,5\},1)$-GDF for any $\{n,m\}\in\{\{5,45\},\{7,23\},\{7,29\},$ $ \{7,35\}, \{9,35\},\{11,19\},\{11,27\}, \{13,17\},\{13,21\},\{13,23\},\{13,27\},\{13,29\},\{13,35\},~ \{15,$ \\ $ 17\}, \{15, 21\},\{15, 27\},\{17,21\},\{17, 27\},\{21, 21\},\{21, 27\},\{23, 45\},\{27,27\}, \{29,45\}, \{35,45\}\}$.
\end{lemma}
{\Proof} When $\{n,m\}= \{7,23\}$, there exists an $(n\times m,\{3,4,5\},1)$-GDF by Theorem \ref{result35}.

When $\{n,m\}\in \{\{15,17\},\{17,21\},\{17,27\}\}$,   a  $(17,3,1)$-GDP with leave $[5]^{2}$ exists by Appendix D;    a PDM$(3,k)$, i.e., a $3$-SPGDD of type $k^{3}$ exists by Lemma \ref{345PDM} (1) and   a $(5\times k,\{3,4,5\},1)$-GDF exists \cite{Su} for $k\in \{15,21,27\}$.
 Applying Constructions \ref{N-NxM-GDP}  and \ref{budong}, there exists an $(n\times m,\{3,4,5\},1)$-GDF.

   When $\{n,m\}\in \{\{23,45\},\{29,45\},\{35,45\}\}$, there is a $(23,3,1)$-GDP with leave $[5]^{2}$, a $(29,3,1)$-GDP with leave $[5]^{3}$ and a $(35,3,1)$-GDP with leave $[5]^{3}$ by Appendix D;    a PDM$(3,45),$ i.e., a $3$-SPGDD of type $45^{3}$ by Lemma \ref{345PDM} (1) and  a $(5\times 45,\{3,4,5\},1)$-GDF by Appendix E.  Applying Constructions \ref{N-NxM-GDP}  and \ref{budong}, there exists an $(n\times m,\{3,4,5\},1)$-GDF.

   When $\{n,m\}\in\{\{5,45\},\{7,29\},\{7,35\},\{9,35\},\{11,19\},\{11,27\},   \{13,17\},\{13,21\},\{13,$ $23\},\{13,27\},\{13,29\}, \{13,35\},\{15, 21\},\{15, 27\},\{21, 21\},\{21, 27\},\{27, 27\}\}$, there exists an $(n\times m,\{3,4,5\},1)$-GDF by Appendix E. \qed\vspace{0.2cm}

\noindent {\bf \emph{Proof of Theorem} \ref{xinresult345}}
Combining the results of Lemmas \ref{result345}, \ref{non345} and \ref{on345},  the theorem is proved.\qed

\section{Applications}
Geometric orthogonal codes were first introduced by Doty and Winslow \cite{Doty}, who  used geometric orthogonal codes to design sets of macrobonds in $3$D DNA origami so as to reduce undesirable bonding arising from misalignment and mismatches.
Rothemund \cite{Rothemund} indicated that when some external conditions,  such as temperature, salinity, density and concentration are the same, compared with a single bonds strength, different bonds strengths (codewords weights) improve the efficiency of constructing nanostructures. Therefore,  it is very important to study variable-weight GOCs.

Let $\mathbb{Z}$ be the set of integers and $n, m, \lambda$ be positive integers. An {\em  $(n\times m,K,\lambda)$-geometric orthogonal code}, briefly  $(n\times m,K,\lambda)$-GOC, is a collection $\cal{C}$ of subsets (called \emph{codewords} or \emph{macrobonds}) of $[0,n-1]\times[0,m-1]$ of sizes from a set of positive integers $K$ such that:
\begin{itemize}
\item[$(1)$] (the aperiodic auto-correlation):
$|B\cap(B+(s,t))|\leq\lambda$ for all $B\in\cal{C}$ and every $(s,t)\in \mathbb{Z}\times \mathbb{Z}\setminus\{(0,0)\}$;
\item[$(2)$] (the aperiodic cross-correlation):
$|A\cap(B+(s,t))|\leq\lambda$ for all $A,B\in\cal{C}$ with $A\neq B$ and every $(s,t)\in\mathbb{Z}\times \mathbb{Z}$,
\end{itemize}
where $B+(s,t)=\{(x+s,y+t):(x,y)\in B\}$. Such GOCs are called \emph{constant-weight} when $|K|=1$ or \emph{variable-weight} when $|K|>1$.

Let $\mathcal{C}$ be a collection of subsets of $[0,n-1]\times[0,m-1]$ of sizes from $K$. It is not convenient to check
the correctness of Conditions $(1)$ and $(2)$. But fortunately, the difference method is very efficient to describe an $(n\times m,K,\lambda)$-GOC for $\lambda=1$. Let $\Delta B$ be the list of differences from $B$ for any $B\in\mathcal{C}$ and $\Delta\mathcal{C}=\bigcup_{B\in\mathcal{C}}\Delta B$. It is readily checked that $\mathcal{C}$ constitutes an $(n\times m,K,1)$-GOC if
$\Delta\mathcal{C}$ covers every element of $[2n-1]\times[2m-1]$ at most once.
Furthermore, $\mathcal{C}$ is called a {\em perfect } $(n\times m,K,1)$-GOC if $\Delta\mathcal{C}=[2n-1]\times[2m-1]\backslash\{(0,0)\}$.

\begin{lemma}\label{GOC-GDF}{\rm (Su  et al. \cite{Su})}
A perfect $(n\times m,K,1)$-GOC is equivalent to an $((2n-1)\times(2m-1),K,1)$-GDF.
\end{lemma}

By Theorems \ref{result35}-\ref{xinresult345} and Lemma \ref{GOC-GDF}, we get the results of perfect GOCs as follows.

\begin{theorem}\label{GOC35-result}
There exists a  perfect $(n\times m,\{3,5\},1)$-GOC if and only if $n, m\equiv0,1\pmod{4}$, $\{n,m\}\neq\{1,p\}$, where~$p\in \{5,8,9,12,21,24\}$.
\end{theorem}

\begin{theorem}\label{GOC36-result}
There exists a   perfect $(n\times m,\{3,6\},1)$-GOC if and only if $n,m\equiv0$ or $1\pmod{3}$, $\{n,m\}\notin\{\{1,p\},\{3,3\}\}$, where~$p\in \{7,10,19\}$.
\end{theorem}

\begin{theorem}\label{GOC345-result}
There exists a   perfect $(n\times m,\{3,4,5\},1)$-GOC if and only if $n,m \neq 2$ and $\{n,m\}\notin \{\{1,p\},\{3,4\},\{3,5\},\{3,7\}\}$, where $p\in \{3,5,6,8,9,11,12,14,15,18,21,24, 27\}$.
\end{theorem}

\section{Concluding remarks}

Based on the equivalence between variable-weight perfect GOCs and   GDFs,  $(n\times m,K,1)$-GDFs were investigated in this paper. By using some auxiliary designs and several recursive constructions, the existence of an $(n\times m, \{3,5\}, 1)$-GDF and an $(n\times m, \{3,6\}, 1)$-GDF  is given. And the remaining parameters of $(n\times m, \{3,4,5\}, 1)$-GDFs have also been completely resolved. As consequences, the  corresponding  perfect $(n\times m,\{3,5\},1)$-GOCs, $(n\times m,\{3,6\},1)$-GOCs and $(n\times m,\{3,4,5\},1)$-GOCs are established.

For variable-weight $(n\times m,K,1)$-GOC, $\mathcal{C}$, if we limit that: there are exactly $q_{i}|\mathcal{C}|$ codewords of weight $k_{i}$, i.e., $q_{i}$ indicates the fraction of codewords of weight $k_{i}$, where $K = \{k_{1}, k_{2}, \ldots, k_{s}\}$, $Q = \{q_{1}, q_{2}, \ldots, q_{s}\}$, and $\sum _{i=1}^{s}q_{i}=1$. The kind of GOCs with this proportion is also a problem worth exploring.

\vspace{0.5cm}

\newpage
{\centering\title{\bf\Large{ Appendix}}}
\appendix
\section{ $(n,\{3,5\},1)$-GDFs with small parameters}\label{app-A}
The  parameters and the base blocks  are listed below.\\

\noindent1.
$n=65$\vspace{0.22cm}
{\scriptsize$$
\begin{array}{llllllll}
\hspace{-2.5cm}\{0, 13, 27\},&\{0, 16, 31\},&\{0, 10, 28\},&\{0, 22, 26\},&\{0, 20, 21,29,32\},&\{0, 6, 23,25,30\}.
\end{array}$$}
\noindent2.
$n=71$\vspace{-0.12cm}
{\scriptsize$$\begin{array}{llllllll}
\hspace{-1.5cm}\{0, 1, 35\},&\{0, 14, 32\},&\{0, 22, 33\},&\{0, 5, 25\},&\{0, 4, 27\},&\{0, 9, 17,24,30\},&\{0, 2, 12,28,31\}.
\end{array}$$}
\noindent3.
$n=89$\vspace{-0.02cm}
{\scriptsize$$\begin{array}{llllllll}
\hspace{-1.2cm}\{0, 1, 44\},&\{0, 2, 40\},&\{0, 12, 35\},&\{0, 17, 41\},&\{0, 22, 37\},&\{0, 4, 14,30,33\},&\{0, 8, 21,28,39\},\\
\hspace{-1.2cm}\{0, 6, 42\},&\{0, 9, 34\}, &\{0, 27, 32\}.
\end{array}$$}
\noindent4.
$n=95$\vspace{0.24cm}
{\scriptsize$$\begin{array}{llllllll}
\hspace{-2.7cm}\{0, 1, 46\},& \{0, 4, 43\},& \{0, 11, 36\},& \{0, 28, 44\},& \{0, 12, 18,31,41\},& \{0, 14, 21,38,47\},\\
\hspace{-2.7cm}\{0, 2, 42\},&\{0, 3, 37\},&\{ 0, 22, 30\},&\{0, 15, 35\},& \{0, 5, 32\}.
\end{array}$$}
\noindent5.
$n=113$\vspace{-0.1cm}
{\scriptsize$$\begin{array}{lllllll}
\hspace{-1.8cm}\{0, 1, 56\},&\{0, 2, 52\},&\{0, 8, 43\},&\{0, 17, 46\},&\{0, 30, 36\},&\{0, 19, 34\},&\{0, 18, 25,39,51\},\\
\hspace{-1.8cm}\{0, 4, 53\},&\{0, 3, 47\},&\{0, 5, 42\},&\{0, 20, 48\},&\{0, 11, 38\},&\{0, 16, 40\},&\{0, 13, 23,45,54\}.\\
\end{array}$$}
\noindent6.
$n=119$\vspace{-0.12cm}
{\scriptsize$$\begin{array}{lllllll}
\hspace{-1.55cm}\{0, 1, 58\},&\{0, 13, 41\},&\{0, 4, 55\},&\{0, 31, 50\},&\{0, 8, 56\},&\{0, 18, 42\},& \{0, 5, 16,39,45\},\cr \hspace{-1.55cm}\{0, 2, 54\},&\{0, 30, 44\},&\{0, 3, 49\},&\{0, 21, 38\},&\{0, 36, 43\},&\{0, 33, 53\},&\{0, 22, 32,47,59\},\cr \hspace{-1.55cm}\{0, 9, 35\}.
\end{array}$$}

\noindent7.
$n=137$\vspace{-0.12cm}
{\scriptsize$$\begin{array}{lllllll}
\hspace{-1.55cm}\{0, 1, 68\},&\{0, 2, 58\},&\{0, 3, 62\},&\{0, 4, 50\},&\{0, 5, 66\},&\{0, 6, 42\},& \{0, 18,25,39,51\},\cr \hspace{-1.55cm}\{0, 29, 63\},&\{0, 15, 53\},&\{0, 43, 60\},&\{0, 8, 52\},&\{0, 24, 64\},&\{0, 28, 47\},&\{0, 13, 23,45,54\},\cr \hspace{-1.55cm}\{0, 16, 65\},&\{0,37,48\},&\{0,35,55\},&\{0,30,57\}.
\end{array}$$}

\noindent8.
$n=143$\vspace{-0.12cm}
{\scriptsize$$\begin{array}{lllllll}
\hspace{-1.55cm}\{0, 1, 70\},&\{0, 10, 42\},&\{0, 4, 67\},&\{0, 3, 61\},&\{0, 6, 68\},&\{0, 11, 57\},& \{0, 13, 30,51,65\},\cr \hspace{-1.55cm}\{0, 19, 60\},&\{0, 15, 54\},&\{0, 2, 66\},&\{0,7,40\},&\{0, 18, 44\},&\{0, 16, 45\},&\{0, 12, 34,43,71\},\cr \hspace{-1.55cm}\{0, 20, 56\},&\{0,23,50\},&\{0,5,53\},&\{0,8,55\},&\{0, 24, 49\}.
\end{array}$$}

\noindent9.
$n=167$\vspace{-0.12cm}
{\scriptsize$$\begin{array}{lllllll}
\hspace{-1.45cm}\{0, 1, 82\},&\{0, 2, 78\},&\{0, 32, 66\},& \{0, 23, 68\},&\{0, 6, 80\},&\{0, 40, 83\},&\{0, 38, 55\},\cr
\hspace{-1.45cm}\{0, 7, 69\},&\{0, 9, 50\},&\{0, 10, 63\},&\{0, 12, 64\},&\{0, 11, 47\},&\{0, 42, 72\},&\{0, 28, 46,61,77\},\cr \hspace{-1.45cm}\{0, 5, 65\},&\{ 0, 4, 79\},&\{0, 25, 54\},&\{0, 21, 56\},&\{0, 26, 48\},&\{0, 39, 58\},&\{0, 20,44,57, 71\},\cr \hspace{-1.45cm}\{0, 8, 67\},&\{0, 3, 73\}.
\end{array}$$}

\noindent10.
$n=185$\vspace{-0.12cm}
{\scriptsize$$\begin{array}{lllllll}
\hspace{-1.45cm}\{0, 1, 91\},&\{0, 2, 87\},&\{0, 3, 68\},& \{0, 4, 86\},&\{0, 5, 81\},&\{0, 6, 78\},&\{0, 9, 89\},\cr
\hspace{-1.45cm}\{0, 8, 83\},&\{0, 7, 63\},&\{0, 10, 60\},&\{0, 11, 53\},&\{0, 12, 59\},&\{0, 19, 74\},&\{0, 28, 46,61,77\},\cr \hspace{-1.45cm}\{0, 21, 64\},&\{ 0, 35, 73\},&\{0, 29, 70\},&\{0, 52, 69\},&\{0, 67, 92\},&\{0, 39, 62\},&\{0, 20,44,57, 71\},\cr \hspace{-1.45cm}\{0, 45, 79\},&\{0, 22, 54\},&\{0,26,66\},&\{0,30,88\},&\{0,36,84\}.
\end{array}$$}

\noindent11.
$n=191$\vspace{-0.12cm}
{\scriptsize$$\begin{array}{lllllll}
\hspace{-1.45cm}\{0, 1, 94\},&\{0, 2, 91\},&\{0, 3, 86\},& \{0, 5, 87\},&\{0, 6, 90\},&\{0, 4, 78\},&\{0, 7, 75\},\cr
\hspace{-1.45cm}\{0, 8, 72\},&\{0, 10, 80\},&\{0, 9, 63\},&\{0, 11, 52\},&\{0, 19, 58\},&\{0, 21, 66\},&\{0, 28, 46,61,77\},\cr \hspace{-1.45cm}\{0, 12, 81\},&\{ 0, 25, 59\},&\{0, 17, 67\},&\{0, 38, 73\},&\{0, 32, 92\},&\{0, 47, 76\},&\{0, 20,44,57, 71\},\cr \hspace{-1.45cm}\{0, 36, 62\},&\{0, 48, 88\},&\{0,43,65\},&\{0,30,85\},&\{0,56,79\},&\{0,53,95\}.
\end{array}$$}

\noindent12.
$n=209$\vspace{-0.12cm}
{\scriptsize$$\begin{array}{lllllll}
\hspace{-1.45cm}\{0, 1, 103\},&\{0, 2, 100\},&\{0, 3, 96\},& \{0, 5, 104\},&\{0, 6, 94\},&\{0, 4, 87\},&\{0, 7, 86\},\cr
\hspace{-1.45cm}\{0, 8, 68\},&\{0, 10, 82\},&\{0, 9, 89\},&\{0, 11, 70\},&\{0, 17, 91\},&\{0, 23, 81\},&\{0,35,76\},\cr \hspace{-1.45cm}\{0, 26, 65\},&\{ 0, 19, 69\},&\{0, 22, 84\},&\{0, 30, 66\},&\{0, 25, 73\},&\{0, 38, 78\},&\{0, 20,44,57, 71\},\cr \hspace{-1.45cm}\{0, 29, 92\},&\{0, 45, 97\},&\{0,47,101\},&\{0,42,95\},&\{0,21,85\},&\{0,34,90\},&\{0, 28, 46,61,77\},\cr \hspace{-1.45cm}\{0,32,75\},&\{0,12,67\}.
\end{array}$$}

\noindent13.
$n=215$\vspace{-0.12cm}
{\scriptsize$$\begin{array}{lllllll}
\hspace{-1.45cm}\{0, 1, 106\},&\{0, 2, 101\},&\{0, 3, 107\},& \{0, 4, 100\},&\{0, 5, 103\},&\{0, 6, 90\},&\{0, 7, 86\},\cr
\hspace{-1.45cm}\{0, 8, 102\},&\{0, 9, 81\},&\{0, 10, 83\},&\{0, 11, 87\},&\{0, 13, 67\},&\{0, 12, 69\},&\{0,27,78\},\cr \hspace{-1.45cm}\{0, 14, 74\},&\{ 0, 26, 89\},&\{0, 19, 64\},&\{0, 20, 52\},&\{0, 15, 77\},&\{0, 33, 91\},&\{0, 30,47,70, 95\},\cr \hspace{-1.45cm}\{0, 28, 71\},&\{0, 35, 85\},&\{0,24,66\},&\{0,36,80\},&\{0,22,68\},&\{0,41,97\},&\{0, 38, 59,75,93\},\cr \hspace{-1.45cm}\{0,31,92\},&\{0,29,82\},&\{0,39,88\}.
\end{array}$$}

\noindent14.
$n=233$\vspace{-0.12cm}
{\scriptsize$$\begin{array}{lllllll}
\hspace{-1.45cm}\{0, 1, 115\},&\{0, 3, 116\},&\{0, 2, 108\},&\{0, 4, 101\},&\{0,5, 107\},&\{0, 6,109\},&\{0, 7, 94\},\cr
\hspace{-1.45cm}\{0, 8, 91\},&\{0, 9, 73\},&\{0, 20, 77\},&\{0, 13, 74\},&\{0, 10, 90\},&\{0, 15, 78\},&\{0, 56, 89\},\cr \hspace{-1.45cm}\{0, 19,  51\},&\{0, 26, 79\},&\{0, 31, 98\},&\{0, 50, 96\},&\{0, 14, 86\},&\{0, 22, 84\},&\{0, 30,47,70, 95\},\cr \hspace{-1.45cm}\{0, 28, 88\},&\{0, 24, 82\},&\{0, 27, 76\},&\{0, 29, 81\},&\{0, 41, 85\},&\{0, 35, 71\},&\{0, 38, 59,75,93\},\cr \hspace{-1.45cm}\{0, 68, 110\},&\{0, 54, 99\},&\{0, 69, 112\},&\{0, 39, 105\},&\{0, 11, 111\},&\{0, 12, 104\}.
\end{array}$$}

\noindent15.
$n=239$\vspace{-0.12cm}
{\scriptsize$$\begin{array}{lllllll}
\hspace{-1.45cm}\{0, 1, 119\},&\{0, 2, 115\},&\{0, 4, 116\},&\{0, 3, 106\},&\{0, 6, 108\},&\{0,7, 111\},&\{0, 8, 117\},\cr
\hspace{-1.45cm}\{0, 5, 84\},&\{0, 9, 94\},&\{0, 10, 98\},&\{0, 11, 89\},&\{0, 13, 90\},&\{0, 14, 97\},&\{0, 60, 99\},\cr \hspace{-1.45cm}\{0, 15, 67 \},&\{0, 12, 66\},&\{0, 19, 62\},&\{0, 20, 76\},&\{0, 22, 80\},&\{0, 27, 91\},&\{0, 69, 110\},\cr \hspace{-1.45cm}\{0, 24, 87\},&\{0, 42, 86\},&\{0, 32, 81 \},&\{0, 31, 92\},&\{ 0, 45, 96\},&\{0, 50, 107\},&\{0, 30,47,70, 95\},\cr \hspace{-1.45cm}\{0, 53, 82\},&\{0, 36, 71\},&\{0, 72, 105\},&\{0, 26, 100\},&\{0, 46, 114\},&\{0, 73, 101\},&\{0, 38, 59,75,93\}.
\end{array}$$}

\noindent16.
$n=257$\vspace{-0.12cm}
{\scriptsize$$\begin{array}{lllllll}
\hspace{-1.45cm}\{0, 1, 128\},&\{0, 2, 124\},&\{0, 4, 125\},&\{0, 3, 117\},&\{0, 6, 119\},&\{0, 5, 102\},&\{0, 8, 126\},\cr
\hspace{-1.45cm}\{0, 7, 99\},&\{0, 9, 110\},&\{0, 35, 94\},&\{0, 11, 84\},&\{0, 14, 85\},&\{0, 12, 86\},&\{0, 60, 123\},\cr \hspace{-1.45cm}\{0, 20, 46\},&\{ 0, 29, 96\},&\{0, 16, 68\},&\{0, 13, 93\},&\{0, 51, 88\},&\{0, 19, 77\},&\{0, 56, 120\},\cr \hspace{-1.45cm}\{0, 65, 90\},&\{0, 31, 76\},&\{0, 36, 78\},&\{0, 47, 91\},&\{0, 30, 83\},&\{0, 33, 95\},&\{0, 40,57,89, 112\},\cr \hspace{-1.45cm}\{0, 15, 115 \},&\{0, 24, 105\},&\{0, 69, 107\},&\{0, 34, 116\},&\{0, 54, 104\},&\{0, 27, 106\},&\{0, 48, 66,87,109\},\cr \hspace{-1.45cm}\{0, 10, 108\},&\{0, 41, 111\},&\{0, 28, 103\}.
\end{array}$$}

\noindent17.
$n=263$\vspace{-0.12cm}
{\scriptsize$$\begin{array}{lllllll}
\hspace{-1.45cm}\{0, 1, 131\},&\{0, 2, 128\},&\{0, 3, 124\},&\{0, 4, 129\},&\{0, 5, 122\},&\{0, 7, 118\},&\{0, 6, 120\},\cr
\hspace{-1.45cm}\{0, 8, 102\},&\{0, 9, 99\},&\{0, 12, 75\},&\{0, 14, 76\},&\{0, 69, 97\},&\{0, 37, 88\},&\{0, 60, 116\},\cr \hspace{-1.45cm}\{0, 41, 86\},&\{0, 47, 85\},&\{0, 24, 82\},&\{ 0, 16, 83\},&\{0, 20, 93\},&\{0, 34, 70\},&\{0, 15, 113\},\cr \hspace{-1.45cm}\{0, 50, 96\},&\{0, 27, 80\},&\{0, 29, 81\},&\{0, 35, 79\},&\{0, 33, 104\},&\{0, 25, 103\},&\{0, 40,57,89, 112\},\cr \hspace{-1.45cm}\{0, 13, 105\},&\{0, 59, 123\},&\{0, 42, 110\},&\{0, 26, 100\},&\{0, 10, 101\},&\{0, 11, 106\},&\{0, 48, 66,87,109\},\cr \hspace{-1.45cm}\{0, 30, 107\},&\{0, 19, 127\},&\{0, 31, 115\},&\{0, 54, 119\}.
\end{array}$$}

\noindent18.
$n=281$\vspace{-0.12cm}
{\scriptsize$$\begin{array}{lllllll}
\hspace{-1.45cm}\{0, 1, 139\},&\{0, 4, 140\},&\{0, 2, 128\},&\{0, 3, 132\},&\{0, 5, 121\},&\{0, 7, 122\},&\{0, 6, 124\},\cr
\hspace{-1.45cm}\{0, 8, 125\},&\{0, 9, 103\},&\{0, 29, 82\},&\{0, 35, 89\},&\{0, 11, 99\},&\{0, 20, 79\},&\{0, 43, 77\},\cr \hspace{-1.45cm}\{0, 27, 91\},&\{0, 16, 97\},&\{0, 24, 61\},&\{0, 31, 78\},&\{0, 28, 80\},&\{0, 17, 90\},&\{0, 21, 86\},\cr \hspace{-1.45cm}\{0, 25, 100\},&\{0, 66, 106\},&\{0, 19, 93\},&\{0, 23, 119\},&\{0, 22, 135\},&\{0, 12, 110\},&\{0, 48, 131\},\cr \hspace{-1.45cm}\{0, 13, 127\},&\{0, 10, 133\},&\{0, 39, 95\},&\{0, 49, 104 \},&\{0, 57, 120\},&\{0, 38, 108\},&\{0, 58, 72,102,134\},\cr \hspace{-1.45cm}\{0, 67, 109\},&\{0, 71, 112\},&\{0, 45, 105\},&\{0, 84, 130\},&\{0, 85, 111\},&\{0, 15, 107\},&\{0, 50,68,101, 137\}.
\end{array}$$}

\noindent19.
$n=287$\vspace{-0.12cm}
{\scriptsize$$\begin{array}{lllllll}
\hspace{-1.45cm}\{0, 1, 142\},&\{0, 3, 143\},&\{0, 4, 136\},&\{0, 2, 139\},&\{0, 5, 127\},&\{0, 6, 129\},&\{0, 7, 115\},\cr
\hspace{-1.45cm}\{0, 8, 126\},&\{0, 9, 120\},&\{0, 33, 93\},&\{0, 37, 81\},&\{0, 21, 55\},&\{0, 24,  91\},&\{0, 23, 86\},\cr \hspace{-1.45cm}\{0, 13, 95\},&\{0, 16, 99\},&\{0, 25, 96\},&\{0, 38, 89\},&\{0, 45, 98\},&\{0, 28, 77\},&\{0, 32, 84\},\cr \hspace{-1.45cm}\{0, 10, 117\},&\{0, 12, 128\},&\{0, 19, 133\},&\{0, 18, 130\},&\{0, 17, 85\},&\{0, 22, 109\},&\{0, 65, 119\},\cr \hspace{-1.45cm}\{0, 42, 121\},&\{0, 11, 105\},&\{0, 36, 102\},&\{0, 43, 113\},&\{0, 29, 101\},&\{0, 27, 124\},&\{0, 50,64,90, 138\},\cr \hspace{-1.45cm}\{0, 35, 110\},&\{0, 92, 131\},&\{0, 57, 135\},&\{0, 41, 103\},&\{0, 47, 106\},&\{0, 20, 100\},&\{0, 58, 73,104,134\},\cr \hspace{-1.45cm}\{0, 56, 125\}.
\end{array}$$}

\noindent20.
$n=305$\vspace{-0.12cm}
{\scriptsize$$\begin{array}{lllllll}
\hspace{-1.45cm}\{0, 1, 152\},&\{0, 3, 150\},&\{0, 2, 140\},&\{0, 5, 144\},&\{0, 4, 130\},&\{0, 7, 132\},&\{0, 8, 136\},\cr
\hspace{-1.45cm}\{0, 6, 137\},&\{0, 9, 121\},&\{0, 44, 92\},&\{0, 28, 58\},&\{0, 47, 97\},&\{0, 29, 72\},&\{0, 34, 95\},\cr \hspace{-1.45cm}\{0, 17, 104\},&\{0, 16, 134\},&\{0, 18, 124\},&\{0, 20, 133\},&\{0, 19, 81\},&\{0, 21, 84\},&\{0, 33, 99\},\cr \hspace{-1.45cm}\{0, 25, 110\},&\{0, 22, 115\},&\{0, 26, 116\},&\{0, 23, 117\},&\{0, 27, 91\},&\{0, 56, 96\},&\{0, 31, 86\},\cr \hspace{-1.45cm}\{0, 12, 141\},&\{0, 10, 108\},&\{0, 11, 120\},&\{0, 14, 114\},&\{0, 52, 119\},&\{0, 69, 145\},&\{0, 60,73,111, 148\},\cr \hspace{-1.45cm}\{0, 42, 122\},&\{0, 41, 146\},&\{0, 89, 143\},&\{0, 65, 101\},&\{0, 77, 123\},&\{0, 32, 103\},&\{0, 68, 83,107,142\},\cr \hspace{-1.45cm}\{0, 70, 149\},&\{0, 82, 127\},&\{0, 78, 135\},&\{0, 53, 102\}.
\end{array}$$}

\noindent21.
$n=311$\vspace{-0.12cm}
{\scriptsize$$\begin{array}{lllllll}
\hspace{-1.45cm}\{0, 1, 154\},&\{0, 2, 152\},&\{0, 4, 155\},&\{0, 3, 141\},&\{0, 5, 145\},&\{0, 6, 130\},&\{0, 7, 133\},\cr
\hspace{-1.45cm}\{0, 9, 132\},&\{0, 8, 143\},&\{0, 35, 89\},&\{0, 20, 82\},&\{0, 40, 96\},&\{0, 24, 66\},&\{0, 26, 95\},\cr \hspace{-1.45cm}\{0, 50, 91\},&\{0, 38, 86\},&\{0, 63, 97\},&\{0, 23, 88\},&\{0, 17, 102\},&\{0, 25, 146\},&\{0, 33, 112\},
\cr \hspace{-1.45cm}\{0, 10, 137\},&\{0, 21, 111\},&\{0, 22, 128\},&\{0, 19, 100\},
&\{0, 37, 136\},&\{0, 13, 117\},&\{0, 45, 119\},\cr \hspace{-1.45cm}\{0, 12, 134\},&\{0, 30, 105\},&\{0, 18, 101\},&\{0, 53, 147\},&\{0, 49, 107\},&\{0, 39, 110\},&\{0, 61,93,120, 148\},\cr \hspace{-1.45cm}\{0, 43, 116\},&\{0, 52, 129\},&\{0, 67, 113\},&\{0, 11, 109\},&\{0, 70, 142\},&\{0, 36, 114\},&\{0, 68, 84,115,144\},\cr \hspace{-1.45cm}\{0, 64, 108\},&\{0, 15, 118\},&\{0, 51, 131\},&\{0, 14, 139\},&\{0, 92, 149\}.
\end{array}$$}

\noindent22.
$n=329$\vspace{-0.12cm}
{\scriptsize$$\begin{array}{lllllll}
\hspace{-1.45cm}\{0, 1, 164\},&\{0, 2, 160\},&\{0, 3, 154\},&\{0, 4, 159\},&\{0, 6, 153\},&\{0, 5, 137\},&\{0, 8, 144\},\cr
\hspace{-1.45cm}\{0, 7, 130\},&\{0, 9, 150\},&\{0, 31, 95\},&\{0, 25, 78\},&\{0, 28, 80\},&\{0, 29, 96\},&\{0, 36, 81\},\cr \hspace{-1.45cm}\{0, 27, 93\},&\{0, 43, 97\},&\{0, 18, 131\},&\{0, 19, 133\},&\{0, 20, 128\},&\{0, 21, 156\},&\{0, 17, 106\},\cr \hspace{-1.45cm}\{0, 22, 101\},&\{0, 23, 125\},&\{0, 26, 126\},&\{0, 11, 140\},&\{0, 12, 122\},
&\{0, 16, 120\},
&\{0, 30, 112\},\cr \hspace{-1.45cm}\{0, 32, 109\},
&\{0, 10, 162\},&\{0, 14, 138\},&\{0, 33, 119\},&\{0, 40, 116\},&\{0, 34, 121\},&\{0, 46, 118\},\cr \hspace{-1.45cm}\{0, 50, 115\},&\{0, 49, 143\},&\{0, 58, 157\},&\{0, 48, 117\},&\{0, 70, 161\},&\{0, 42, 105\},&\{0, 60,73,111, 148\},\cr \hspace{-1.45cm}\{0, 71, 127\},&\{0, 98, 145\},&\{0, 92, 149\},&\{0, 55, 139\},&\{0, 85, 146\},&\{0, 41, 103\},&\{0, 68, 83,107,142\},\cr \hspace{-1.45cm}\{0, 90, 134\}.
\end{array}$$}

\noindent23.
$n=335$\vspace{-0.12cm}
{\scriptsize$$\begin{array}{lllllll}
\hspace{-1.45cm}\{0, 1, 167\},&\{0, 2, 163\},&\{0, 3, 158\},&\{0, 5, 159\},&\{0, 6, 162\},&\{0, 4, 145\},&\{0, 7, 150\},\cr
\hspace{-1.45cm}\{0, 9, 151\},&\{0, 8, 133\},&\{0, 28, 97\},&\{0, 32, 92\},&\{0, 23, 91\},&\{0, 44, 96\},&\{0, 37, 99\},\cr \hspace{-1.45cm}\{0, 35, 94\},&\{0, 17, 126\},&\{0, 20, 135\},&\{0, 18, 132\},&\{0, 19, 106\},&\{0, 22, 111\},&\{0, 47, 118\},\cr \hspace{-1.45cm}\{0, 25, 130\},&\{0, 21, 140\},
&\{0, 13, 129\},&\{0, 24, 101\},&\{0, 27, 102\},&\{0, 11, 164\},&\{0, 10, 117\},\cr \hspace{-1.45cm}\{0, 33, 131\},&\{0, 34, 146\},&\{0, 57, 100\},&\{0, 36, 103\},&\{0, 16, 160\},&\{0, 42,120\},&\{0, 38, 148\},\cr \hspace{-1.45cm}\{0, 45, 108\},&\{0, 55, 121\},&\{0, 86, 137\},&\{0, 79, 128\},&\{0, 95, 149\},&\{0, 65, 147\},&\{0, 50,64,90, 138\},\cr \hspace{-1.45cm}\{0, 70, 123\},&\{0, 12, 139\},&\{0, 85, 157\},&\{0, 39, 152\},&\{0, 56, 136\},&\{0, 93, 122\},&\{0, 58, 73,104,134\},\cr \hspace{-1.45cm}\{0, 84, 165\},&\{0, 41, 124\}.
\end{array}$$}

\noindent24.
$n=353$\vspace{-0.12cm}
{\scriptsize$$\begin{array}{lllllll}
\hspace{-1.45cm}\{0, 1, 176\},&\{0, 2, 170\},&\{0, 3, 172\},&\{0, 4, 164\},&\{0, 5, 163\},&\{0, 6, 171\},&\{0, 7, 156\},\cr
\hspace{-1.45cm}\{0, 8, 159\},&\{0, 9, 166\},&\{0, 26, 95\},&\{0, 11, 132\},&\{0, 12, 142\},&\{0, 14, 143\},&\{0, 83, 147\},\cr \hspace{-1.45cm}\{0, 17, 162\},&\{0, 13, 154\},&\{0, 15, 126\},&\{0, 18, 110\},&\{0, 19, 153\},&\{0, 20, 117\},&\{0, 46, 118\},\cr \hspace{-1.45cm}\{0, 22, 155\},&\{0, 23, 105\},&\{0, 21, 108\},&\{0, 24, 140\},&\{0, 25, 125\},&\{0, 28, 167\},&\{0, 10, 138\},
\cr \hspace{-1.45cm}\{0, 33, 104\},&\{0, 34, 114\},&\{0, 35, 101\},&\{0, 36, 113\},&\{0, 37, 107\},&\{0, 38, 112\},&\{0, 39, 124\},\cr \hspace{-1.45cm}\{0, 41, 122\},&\{0, 44, 135\},&\{0, 40, 102\},&\{0, 45, 131\},&\{0, 49, 127\},&\{0, 58, 137\},&\{0, 61,93,120, 150\},\cr \hspace{-1.45cm}\{0, 56, 146\},&\{0, 43, 106\},&\{0, 42, 109\},&\{0, 48, 136\},&\{0, 96, 161\},&\{0, 51, 103\},&\{0, 68, 84,115,144\},\cr \hspace{-1.45cm}\{0, 53, 152\},&\{0, 55, 174\},&\{0, 54, 148\},&\{0, 73, 123\},&\{0, 75, 173\}.
\end{array}$$}

\noindent25.
$n=359$\vspace{-0.12cm}
{\scriptsize$$\begin{array}{lllllll}
\hspace{-1.45cm}\{0, 1, 178\},&\{0, 2, 172\},&\{0, 3, 176\},&\{0, 4, 166\},&\{0, 5, 165\},&\{0, 6, 174\},&\{0, 7, 158\},\cr
\hspace{-1.45cm}\{0, 8, 167\},&\{0, 9, 163\},&\{0, 26, 99\},&\{0, 33, 70\},&\{0, 50, 98\},&\{0, 13, 132\},&\{0, 58, 161\},\cr \hspace{-1.45cm}\{0, 15, 142\},&\{0, 12, 124\},&\{0, 17, 138\},&\{0, 19, 145\},&\{0, 20, 156\},&\{0, 21, 102\},&\{0, 64, 129\},\cr \hspace{-1.45cm}\{0, 18, 114\},&\{0, 23, 118\},&\{0, 22, 175\},&\{0, 24, 135\},&\{0, 25, 110\},&\{0, 30, 122\},&\{0, 34, 147\},\cr \hspace{-1.45cm}\{0, 10, 140\},&\{0, 11, 152\},
&\{0, 35, 109\},&\{0, 36, 105\},&\{0, 38, 100\},&\{0, 42, 143\},&\{0, 39, 117\},\cr \hspace{-1.45cm}\{0, 44, 123\},&\{0, 43, 125\},&\{0, 45, 133\},&\{0, 40, 134\},&\{0, 63, 171\},&\{0, 49, 116\},&\{0, 61,93,120,148\},\cr \hspace{-1.45cm}\{0, 56, 146\},&\{0, 51, 157\},&\{0, 54, 131\},&\{0, 75, 179\},&\{0, 53, 139\},&\{0, 83, 155\},&\{0, 68, 84,115,144\},\cr \hspace{-1.45cm}\{0, 97, 149\},&\{0, 46, 137\},&\{0, 41, 107\},&\{0, 14, 164\},&\{0, 71, 128\},&\{0, 89, 169\}.
\end{array}$$}

\noindent26.
$n=377$\vspace{-0.12cm}
{\scriptsize$$\begin{array}{lllllll}
\hspace{-1.45cm}\{0, 1, 135\},&\{0, 2, 188\},&\{0, 4, 110\},&\{0, 3, 102\},&\{0, 6, 179\},&\{0, 9, 155\},&\{0,7, 187\},\cr
\hspace{-1.45cm}\{0, 5, 124\},&\{0, 73, 81\},&\{0, 71, 175\},&\{0, 56, 176\},&\{0, 42, 142\},&\{0, 46, 172\},&\{0, 45, 97\},\cr \hspace{-1.45cm}\{0, 94, 143\},&\{0, 91, 117\},&\{0, 33, 138\},&\{0, 79, 137\},&\{0, 11, 167\},&\{0, 64, 182\},&\{0, 20, 103\},\cr \hspace{-1.45cm}\{0, 27, 136\},&\{0, 53, 133\},&
 \{0, 13, 154\},&\{0, 76, 148\},&\{0, 39, 140\},&\{0, 60, 121\},&\{0, 82, 145\},\cr \hspace{-1.45cm}\{0, 50, 163\},& \{0, 88, 184\},&
\{0, 35, 162\},&\{0, 15, 131\},&\{0, 84, 150\},&\{0, 22, 151\},&\{0, 67, 157\},\cr \hspace{-1.45cm}\{0, 54, 177\},&\{0, 24, 183\},&
 \{0, 43, 165\},&\{0, 31, 108\},&\{0, 44, 158\},&\{0, 85, 174\},&\{0, 87, 161\},\cr \hspace{-1.45cm}\{0, 28, 139\},&\{0, 69, 147\},&\{0, 68, 125\},&\{0, 51, 166\},&\{0, 10, 170\},&\{0, 98, 112\},&\{0, 30,47,70,95\},\cr \hspace{-1.45cm}\{0, 36, 164\},& \{0, 92, 178\},&
 \{0, 41, 171\},
&\{0, 29, 181\},&\{0, 12, 144\}, &\{0, 32, 185\},&\{0, 38, 59,75,93\},\cr \hspace{-1.45cm}\{0, 62, 169\},&\{0, 19, 168\}.
\end{array}$$}

\noindent27.
$n=383$\vspace{-0.12cm}
{\scriptsize$$\begin{array}{lllllll}
\hspace{-1.45cm}\{0, 2, 149\},&\{0, 4, 125\},&\{0, 7, 188\},&\{0, 6, 141\},&\{0, 3, 100\},&\{0, 1, 163\},&\{0, 8, 99\},\cr
\hspace{-1.45cm} \{0, 5, 161\},& \{0, 9, 116\},&\{0, 15, 182\},&\{0, 26, 180\},&\{0, 87, 179\},&\{0, 60, 177\},&\{0, 57, 81\},\cr \hspace{-1.45cm}\{0, 52, 178\},&
 \{0, 68, 166\},&\{0, 73, 174\},&\{0, 19, 172\},&\{0, 35, 175\},&\{0, 71, 173\},&\{0, 29, 133\},\cr \hspace{-1.45cm}\{0, 58, 171\},&\{0, 10, 169\}, &
\{0, 80, 170\},&\{0, 36, 158\},&\{0, 41, 151\},&\{0, 54, 184\},&\{0, 88, 134\},\cr \hspace{-1.45cm} \{0, 27, 191\},&\{0, 14, 190\},&
\{0, 22, 146\},&\{0, 33, 152\},&\{0, 83, 114\},&\{0, 28, 137\},&\{0, 45, 183\},\cr \hspace{-1.45cm}\{0, 53, 132\},&\{0, 11, 123\}, &
\{0, 13, 142\},&\{0, 49, 131\},&\{0, 69, 111\},&\{0, 43, 186\},&\{0, 61, 155\},\cr \hspace{-1.45cm}\{0, 74, 160\},&\{0, 32, 150\}, &
\{0, 64, 103\},&\{0, 89, 145\},&\{0, 67, 187\},&\{0, 51, 157\},&\{0, 30,47,70,95\},\cr \hspace{-1.45cm}\{0, 12, 148\},&\{0, 84, 128\},&\{0, 77, 185\},
&\{0, 50, 165\},&\{0, 20, 105\},&\{0, 76, 139\},&\{0, 38, 59,75,93\},\cr \hspace{-1.45cm}\{0, 96, 168\},&\{0, 62, 189\},&
 \{0, 66, 144\}.
\end{array}$$}

\noindent28.
$n=401$\vspace{-0.12cm}
{\scriptsize$$\begin{array}{lllllll}
\hspace{-1.45cm}\{0, 1,  148\},&\{0, 2, 127\},&\{0, 3, 139\},&\{0, 4, 196\},&\{0, 6, 182\},&\{0, 5, 167\},&\{0, 8, 168\},\cr
\hspace{-1.45cm} \{0, 7,119\},&\{0, 9, 190\},&\{0, 22, 94\},&\{0, 10, 198\},&\{0, 11, 97\},&\{0, 35, 194\},&\{0, 31,183\},\cr \hspace{-1.45cm}\{0, 53, 179\}, &
\{0, 33, 178\},&\{0, 64,  171\},&\{0, 36, 164\},&\{0, 62, 186\},&\{0, 61, 177\},&\{0, 28, 197\},\cr \hspace{-1.45cm}\{0, 43, 157\},&\{0, 82, 174\}, &
\{0, 32, 166\},&\{0, 42, 185\},&\{0, 20, 138\},&\{0, 80,191\},&\{0, 15, 199\},\cr \hspace{-1.45cm}\{0, 27, 130\},&\{0, 71, 140\}, &
\{0, 88, 132\},&\{0, 77, 129\},&\{0, 49, 155\},&\{0, 51, 172\},&\{0, 26, 135\},\cr \hspace{-1.45cm}\{0, 58, 131\},&\{0, 54, 156\},&\{0, 67, 200\},
&\{0, 78, 195\},&\{0, 96, 153\},&\{0, 60, 123\},&\{0, 13, 163\},\cr \hspace{-1.45cm}\{0, 76,  180\},& \{0,56, 193\},&
\{0, 83, 149\},&\{0, 85, 175\},&\{0, 14, 113\},&\{0, 79, 189\},&\{0, 30,47,70,95\},\cr \hspace{-1.45cm}\{0, 50, 151\},&\{0, 91,120 \}, &
\{0, 98, 144\},&\{0, 39, 154\},&\{0, 68, 173\},&\{0, 45, 187\},&\{0, 38, 59,75,93\},\cr \hspace{-1.45cm}\{0, 87, 161\},&\{0, 12, 170\},&
 \{0, 89, 108\},&\{0, 81, 165\},&\{0, 24, 146\},&\{0, 41, 141\}.
\end{array}$$}

\noindent29.
$n=407$\vspace{-0.12cm}
{\scriptsize$$\begin{array}{lllllll}
\hspace{-1.45cm}\{0, 1,152\},&\{0, 3, 190\},&\{0, 4, 92\},&\{0, 6, 159\},&\{0, 7,157\},&\{0, 9, 115\},&\{0, 5, 126\},\cr
\hspace{-1.45cm}\{0, 8, 109\},&\{0, 2, 110\},& \{0, 20,96\},&\{0, 69, 98\},&\{0, 39,188\},&\{0, 85,192\},&\{0, 81, 154\},\cr \hspace{-1.45cm}\{0, 84, 189\},&
 \{0, 13, 193\},&\{0, 52, 191\},&\{0, 42, 183\},&\{0, 46, 182\},&\{0,79, 181\},&\{0, 35, 203\},\cr \hspace{-1.45cm}\{0, 28,  156\},&\{0, 26, 172\}, &
\{0, 99, 148\},&\{0, 44, 173\},&\{0, 10, 134\},&\{0, 63,143\},&\{0, 74, 177\},\cr \hspace{-1.45cm}\{ 0,72,184\},&\{0, 91,185\}, &
\{0, 50, 170\},&\{0, 15, 179\},&\{0, 11,176\},&\{0, 90, 161\},&\{0, 41, 171\},\cr \hspace{-1.45cm}\{0, 22, 197\},&\{0, 12, 178\},&
\{0, 57, 201\},&\{0, 68, 195\},&\{0, 60,160\},&\{0, 86, 147\},&\{0, 56, 194\},\cr \hspace{-1.45cm}\{0, 62, 199\},&\{0, 36, 158\}, &
\{0, 58, 174\},&\{0, 43, 125\},&\{0,32,  155\},&\{0, 27, 104\},&\{0, 66, 198\},\cr \hspace{-1.45cm}\{0, 19, 133\},&\{0, 64, 142\},&
\{0, 54, 167\},&\{0, 89, 186\},&\{0, 51, 118\},&\{0, 31, 200\},&\{0, 30,47,70,95\},\cr \hspace{-1.45cm}\{0, 24, 135\}, &\{0, 33, 196\},&
\{0, 14, 145\},&\{0, 83, 202\},&\{0, 53, 140\},&\{0, 45, 162\},&\{0, 38, 59,75,93\}.
\end{array}$$}

\noindent30.
$n=425$\vspace{-0.12cm}
{\scriptsize$$\begin{array}{lllllll}
\hspace{-1.45cm}\{0, 1, 108\},&\{0, 2, 186\},&\{0, 9, 121\},&\{0, 8, 156\},&\{0, 6,208\},&\{0, 3, 209\},&\{0, 5, 205\}, \cr
\hspace{-1.45cm} \{0, 7, 173\},&\{0, 4, 178\},&\{0, 60, 189\},&\{0, 97, 201\},&\{0, 39, 193\},&\{0, 90, 203\},&\{0, 79,162\},\cr \hspace{-1.45cm}\{0, 54, 169\}, &
\{0, 43, 195\},&\{0, 58, 196\},&\{0, 52, 179\},&\{0, 13, 188\},&\{0, 53, 212\},&\{0, 28, 204\},\cr \hspace{-1.45cm}\{0, 24, 160\},&\{0, 66, 146\},&
 \{0, 42, 164\},&\{0, 64, 183\},&\{0, 11, 182\},&\{0, 36,168\},&\{0, 69, 142\},\cr \hspace{-1.45cm}\{0, 74, 211\},&\{0, 63, 181\},&
 \{0, 88, 197\},&\{0, 33, 149\},&\{0, 14, 194\},&\{0, 62, 207\},&\{0, 67,192\},\cr \hspace{-1.45cm}\{0, 41,144\},&\{0, 31,151\}, &
\{0, 19, 110\},&\{0, 50, 139\},&\{0, 44, 191\},&\{0, 35, 165\},&\{0, 71, 117\},\cr \hspace{-1.45cm}\{0, 10, 134\},&\{0, 85, 153\},&
 \{0, 27, 167\},&\{0,20, 131\},&\{0, 76, 102\},&\{0, 29, 143\},&\{0, 49, 172\},\cr \hspace{-1.45cm}\{0, 45, 141\},&\{0, 81, 163\},&
 \{0, 12, 210\},
&\{0, 51, 177\},&\{0, 32, 190\},&\{0, 86, 187\},&\{0, 30,47,70,95\},\cr \hspace{-1.45cm}\{0, 78, 170\},&\{0, 98,185\},&
 \{0, 57, 157\},&\{0, 94, 199\},&\{0, 84, 161\},&\{0, 15, 150\},&\{0, 38, 59,75,93\},\cr \hspace{-1.45cm}\{0,72,  133\},&\{0, 22,128\},&
\{0, 99, 155\}.
\end{array}$$}

\noindent31.
$n=431$\vspace{-0.12cm}
{\scriptsize$$\begin{array}{lllllll}
\hspace{-1.45cm}\{0, 1, 178\}, &\{0, 7, 169\},&\{0, 3, 205\},&\{0, 5, 194\},&\{0, 4, 197\},&\{0, 8, 179\},&\{0, 6, 207\},\cr
\hspace{-1.45cm} \{0,2,114\},&\{0, 9, 196\}, &\{0, 10, 79\},&\{0, 24, 96\},&\{0, 12, 203\},&\{0, 42, 164\},&\{0, 68,208\},\cr \hspace{-1.45cm} \{0, 62,204\},&
\{0, 33, 183\},&\{0, 71,175\},&\{0, 78, 172\},&\{0, 64,198\},&\{0, 53,185\},&\{0, 56,155\},\cr \hspace{-1.45cm}\{0, 31, 174\},&\{0, 87, 214\},&
\{0, 92, 215\},&\{0, 52, 213\},&\{0, 74, 181\},&\{0, 36, 212\},&\{0, 32,  195\},\cr \hspace{-1.45cm}\{0, 13, 165\},&\{0, 85, 167\}, &
\{0, 28, 148\},&\{0, 58, 158\},&\{0, 43, 199\},&\{0, 57, 211\},&\{0, 35, 180\},\cr \hspace{-1.45cm}\{0, 20, 138\},&\{0, 39, 116\}, &
\{0, 50, 209\},&\{0, 19, 160\},&\{0, 73, 149\},&\{0, 89,200\},&\{0, 49, 133\},\cr \hspace{-1.45cm}\{0, 83, 186\},&\{0, 22, 147\},&
 \{0, 29, 135\},
 &\{0, 81, 210\},
&\{0, 41,192\},&\{0, 61, 170\},&\{0,63,184\},\cr \hspace{-1.45cm}\{0, 26,136\},&\{0, 54, 173\},&
 \{0, 90, 188\},&\{0, 11,139\},&\{0, 91, 137\},&\{0, 45, 153\},&\{0, 30,47,70,95\},\cr \hspace{-1.45cm}\{0, 60, 126\},&\{0, 67, 182\},&
\{0, 27, 124\},&\{0, 14,144\},&\{0, 88, 190\},&\{0, 51,168\},&\{0, 38, 59,75,93\},\cr \hspace{-1.45cm}\{0,80, 166\},&\{0,15, 146\}, &
\{0, 44, 157\},&\{0, 105, 206\}.
\end{array}$$}

\noindent32.
$n=449$\vspace{-0.12cm}
{\scriptsize$$\begin{array}{lllllll}
\hspace{-1.45cm}\{0, 8, 159\},&\{0, 7, 165\},&\{0, 9, 127\},
 &\{0, 3, 182\},&\{0, 5,211\},&\{0, 6, 150\},&\{0, 4, 191\},\cr
\hspace{-1.45cm}\{0, 1, 130\},&\{0, 2, 124\},&\{0, 58, 212\},&\{0, 29, 214\},&\{0, 20, 180\},&\{0, 81, 215\},&\{0, 52, 189\},\cr \hspace{-1.45cm}\{0,39,208\},&
 \{0, 50, 213\},&\{0, 36, 197\},&\{0, 42, 204\},&\{0, 90, 207\},&\{0, 35, 173\},&\{0,80, 141\},\cr \hspace{-1.45cm}\{0, 85, 174\},&\{0, 13, 223\},&\{0, 41, 222\},&\{0, 15,220\},&\{0, 51, 219\},&\{0, 62,217\},&\{0, 73, 221\},\cr \hspace{-1.45cm}\{0,82, 193\},&\{0, 54,194\}, &
\{0, 31,139\},&\{0,46, 172\},&\{0, 68, 171\},&\{0, 78,149\},&\{0, 11, 195\},\cr \hspace{-1.45cm}\{0, 98, 167\},&\{0,63, 170\},& \{0, 27, 110\},&\{0, 91, 203\},&\{0, 22,136\},&\{0, 12, 106\},&\{0,10,145\},\cr \hspace{-1.45cm}\{0,45, 178\},&\{0, 99, 131\},&
\{0, 43, 190\},&\{0, 44,164\},&\{0, 87,166\},&\{0, 33,216\},&\{0, 92, 196\},\cr \hspace{-1.45cm}\{0, 60, 157\},&\{0, 84, 199\},&
 \{0, 76, 143\},&\{0, 14,200\},&\{0, 56, 198\},&\{0,77, 209 \},&\{0, 28, 156\},\cr \hspace{-1.45cm}\{0, 96, 153\},&\{0, 66, 175\}, &
\{0, 74, 146\},&\{0, 19,121\},&\{0, 88,188\},&\{0,86, 202\},&\{0, 30,47,70,95\},\cr \hspace{-1.45cm}\{0, 53, 176\},&\{0, 26, 218\}, &
\{0, 49, 201\},&\{0, 24, 125\},&\{0, 64, 177\},&\{0, 105,224\},&\{0, 38, 59,75,93\}.
\end{array}$$}

\section{ $(n\times m,\{3,5\},1)$-GDFs with small parameters}\label{app-B}
The parameters and the base blocks  are listed below.\\

\noindent1.~$(n,m)=(7,9):$
{\scriptsize  $$\begin{array}{llll}
\{(1, 3), (3, 0), (0, 2)\},&
\{(2, 0), (0, 4), (2, 3)\},&
\{(1, 0), (0, 4), (3, 3)\},&
\{(0, 0), (3, 2), (3, 4)\},\\
\{(0, 0), (2, 1), (3, 3)\},&
\{(0, 0), (1, 4), (2, 4)\},&
\{(0, 0), (1, 3), (2, 0)\},&
\{(3, 0), (1, 2), (0, 3), (0, 4), (3, 4)\}.
\end{array}$$}

\noindent2.~$(n,m)=(7,15):$
{\scriptsize $$\begin{array}{llll}
 \{(0, 0), (1, 0), (3, 7)\},&
\{(1, 0), (3, 0), (0, 7)\},&
\{(0, 0), (3, 0), (0, 6)\},&
\{(3, 0), (0, 1), (0, 5)\},\\

 \{(2, 0), (1, 1), (0, 7)\},&
\{(3, 0), (1, 1), (0, 4)\},&
\{(2, 0), (3, 1), (0, 4)\},&
\{(3, 0), (0, 2), (1, 6)\},\\

 \{(0, 0), (0, 7), (2, 2)\},&
\{(0, 0), (1, 5), (3, 2)\},&
\{(0, 0), (0, 2), (2, 6)\},&
\{(1, 0), (0, 2), (3, 3),(1,5),(2,7)\}, \\

\{(0, 0), (1, 6), (2, 1)\},&
 \{(0, 0), (0, 3), (3, 6)\},&
\{(0, 0), (0, 1), (3, 5)\}.

\end{array}$$}

\noindent3.~$(n,m)=(7,17)$
{\scriptsize $$\begin{array}{llll}
 \{(0, 0), (1, 0), (3, 8)\},&
\{(1, 0), (3, 0), (0, 8)\},&
\{(0, 0), (0, 7), (3, 0)\},&
\{(3, 0), (0, 1), (3, 8)\},\\

 \{(3, 0), (1, 1), (0, 6)\},&
\{(3, 0), (2, 1), (0, 5)\},&
\{(2, 0), (0, 2), (0, 8)\},&
\{(0, 0), (0, 2), (1, 3),(1,8),(3,5)\},\\

 \{(3, 0), (0, 3), (1, 7)\},&
\{(0, 0), (2, 1), (3, 6)\},&
\{(0, 0), (2, 3), (2, 7)\},&
\{(1, 0), (1, 1), (0, 4),(3,5),(0,7)\},\\

 \{(0, 0), (1, 7), (3, 2)\},&
\{(2, 0), (3, 2), (0, 6)\},&
\{(1, 0), (0, 2), (3, 6)\}.
\end{array}$$}

\noindent4.~$(n,m)=(7,23)$
{\scriptsize $$\begin{array}{llll}
 \{(1, 0), (1, 4), (0, 8)\},&
\{(0, 0), (1, 1), (3, 6)\},&
\{(0, 0), (3, 0), (3, 10)\},&
\{(2, 0), (0, 1), (3, 10)\},\\

 \{(3, 0), (0, 1), (3, 9)\},&
\{(1, 0), (3, 8), (0, 1)\},&
\{(0, 0), (0, 1), (2, 11)\},&
\{(0, 0), (2, 0), (3, 11)\},\\

 \{(0, 0), (2, 1), (3, 5)\},&
\{(2, 0), (3, 6), (0, 3)\},&
\{(0, 5), (3, 0), (1, 10)\},&
\{(2, 0), (3, 8), (0, 7)\},\\

 \{(2, 0), (2, 3), (0, 8)\},&
\{(1, 6), (3, 0), (0, 8)\},&
\{(3, 0), (3, 7), (0, 10)\},&
\{(3, 0), (1, 2), (3, 11),(0,7),(0,9)\},\\

 \{(3, 0), (2, 9), (0, 6)\},&
\{(0, 0), (0, 6), (1, 9)\},&
\{(2, 0), (3, 0), (0, 11)\},&
\{(2, 0), (1, 11), (3, 7),(1,3),(0,9)\},\\

 \{(0, 0), (2, 2), (2, 7)\},&
\{(1, 0), (3, 6), (0, 10)\}.
\end{array}$$}

\noindent5.~$(n,m)=(9,9)$
{\scriptsize $$\begin{array}{llll}
 \{(0, 0), (4, 0), (0, 4)\},&
\{(0, 0), (1, 1), (4, 3)\},&
\{(0, 1), (1, 0), (4, 3)\},&
\{(3, 0), (0, 4), (4, 2)\},\\

 \{(4, 0), (2, 4), (0, 1)\},&
\{(0, 2), (2, 1), (2, 0)\},&
\{(1, 0), (0, 2), (2, 3)\},&
\{(0, 0), (1, 0), (2, 4),(4, 1),(4, 4)\},\\

\{(4, 0), (3, 3), (0, 3)\},&
 \{(0, 3), (3, 0), (3, 2)\},&
\{(1, 0), (0, 4), (3, 2)\}.

\end{array}$$}

\noindent6.~$(n,m)=(9,15)$

{\scriptsize $$\begin{array}{llll}
\{(0, 0), (0, 2), (2, 5),(4,0),(4,6)\},&
 \{(4, 0), (0, 5), (1, 1)\},&
\{(0, 4), (4, 0), (0, 1)\},&
\{(2, 0), (3, 1), (0, 7)\},\\

\{(0, 3), (4, 0), (2, 3),(0,7),(3,0)\},&
 \{(2, 0), (2, 7), (0, 1)\},&
\{(0, 0), (3, 2), (4, 7)\},&
\{(0, 0), (1, 3), (3, 7)\},\\

\{(1, 2), (0, 4), (3, 0),(4,7),(4,6)\},&
\{(3, 0), (1, 6), (0, 2)\},&
\{(0, 0), (3, 6), (4, 1)\},&
\{(0, 0), (1, 2), (4, 5)\},\\

\{(1, 0), (3, 2), (4, 1),(0,7),(3,7)\}.\\

\end{array}$$}

\noindent7.~$(n,m)=(9,17)$
{\scriptsize $$\begin{array}{llll}
 \{(3, 0), (4, 0), (0, 8)\},&
\{(0, 0), (2, 0), (4, 8)\},&
\{(1, 0), (4, 0), (0, 7)\},&
\{(0, 0), (0, 6), (4, 0)\},\\

 \{(4, 8), (4, 0), (0, 1)\},&
\{(4, 0), (0, 5), (1, 1)\},&
\{(3, 0), (0, 7), (1, 1)\},&
\{(2, 0), (2, 1), (0, 8)\},\\

 \{(3, 1), (4, 0), (0, 4)\},&
\{(3, 0), (0, 4), (4, 1)\},&
\{(0, 0), (0, 7), (3, 1)\},&
\{(4, 0), (2, 6), (0, 2)\},\\

 \{(0, 0), (2, 5), (4, 2)\},&
\{(0, 0), (2, 6), (4, 1)\},&
\{(1, 0), (3, 3), (0, 8)\},&
\{(0, 0), (0, 2), (4, 6)\},\\

 \{(2, 0), (3, 8), (0, 4)\},&
\{(4, 7), (3, 0), (0, 2)\},&
\{(0, 0), (3, 3), (3, 8)\},&
\{(1, 0), (3, 7), (0, 5),(0,2),(2,3)\},\\

\{(1, 0), (4, 6), (0, 3)\},&
 \{(0, 0), (2, 2), (3, 7)\},&
\{(0, 0), (1, 2), (1, 6)\}.

\end{array}$$}

\noindent8.~$(n,m)=(9,23)$
{\scriptsize $$\begin{array}{llll}
 \{(0, 0), (1, 2), (4, 7)\},&
\{(0, 5), (2, 0), (4, 1)\},&
\{(0, 0), (2, 0), (4, 11)\},&
\{(0, 0), (4, 0), (4, 10)\},\\

 \{(0, 1), (0, 9), (4, 0)\},&
\{(0, 8), (1, 1), (4, 0)\},&
\{(0, 10), (2, 1), (3, 0)\},&
\{(0, 7), (2, 1), (4, 0)\},\\

 \{(0, 6), (4, 0), (4, 1)\},&
\{(0, 0), (1, 4), (3, 8)\},&
\{(0, 10), (2, 0), (3, 1)\},&
\{(0, 10), (1, 0), (4, 0)\},\\

 \{(0, 0), (1, 8), (4, 1)\},&
\{(0, 4), (1, 0), (4, 1)\},&
\{(0, 2), (1, 11), (2, 0)\},&
\{(0, 2), (0, 8), (3, 0)\},\\

 \{(0, 6), (3, 0), (3, 2)\},&
\{(0, 0), (2, 7), (4, 3)\},&
\{(0, 2), (2, 11), (4, 0)\},&
\{(0, 11), (3, 0), (4, 0)\},\\

 \{(0, 0), (2, 2), (3, 9)\},&
\{(0, 0), (3, 2), (4, 8)\},&
\{(0, 0), (0, 4), (3, 10)\},&
\{(0, 0), (1, 10), (3, 3)\},\\

 \{(0, 0), (0, 9), (1, 3)\},&
\{(0, 3), (1, 0), (4, 7)\},&
\{(0, 0), (0, 11), (2, 3)\},&
\{(0, 0), (0, 5), (2, 10)\},\\

 \{(0, 3), (2, 0), (4, 8)\},&
\{(0, 8), (1, 0), (1, 3)\},&
\{(0, 5), (3, 0), (4, 11)\},&
\{(0, 2), (0, 9), (1, 0), (3, 6), (4, 11)\}.
\end{array}$$}

\noindent9.~$(n,m)=(15,15)$
{\scriptsize $$\begin{array}{llll}
 \{(0, 7), (6, 0), (7, 0)\},&
\{(0, 0), (2, 0), (7, 7)\},&
\{(0, 6), (4, 0), (7, 0)\},&
\{(0, 0), (4, 0), (7, 6)\},\\

 \{(0, 0), (5, 0), (6, 7)\},&
\{(0, 0), (6, 0), (7, 5)\},&
\{(0, 0), (6, 6), (7, 0)\},&
\{(0, 1), (2, 7), (7, 0)\},\\

 \{(0, 1), (5, 0), (7, 5)\},&
\{(0, 1), (1, 0), (5, 7)\},&
\{(0, 1), (6, 0), (7, 4)\},&
\{(0, 1), (2, 0), (6, 6)\},\\

 \{(0, 1), (3, 0), (6, 5)\},&
\{(0, 1), (4, 0), (7, 2)\},&
\{(0, 0), (0, 1), (6, 3)\},&
\{(0, 4), (6, 0), (7, 1)\},\\

 \{(0, 6), (3, 0), (6, 1)\},&
\{(0, 4), (4, 0), (6, 1)\},&
\{(0, 6), (2, 0), (7, 1)\},&
\{(0, 2), (7, 0), (7, 4)\},\\

 \{(0, 6), (1, 2), (5, 0)\},&
\{(0, 2), (5, 0), (5, 6)\},&
\{(0, 6), (3, 2), (6, 0)\},&
\{(0, 2), (2, 0), (2, 5)\},\\

 \{(0, 0), (4, 5), (5, 2)\},&
\{(0, 0), (3, 7), (5, 3)\},&
\{(0, 0), (0, 3), (2, 7)\},&
\{(0, 0), (1, 6), (6, 1)\},\\

 \{(0, 0), (4, 2), (5, 5)\},&
\{(0, 3), (4, 7), (5, 0)\},&
\{(0, 3), (3, 0), (3, 7)\},&
\{(0, 5), (1, 7), (2, 2), (3, 0), (6, 3)\},\\

\{(0, 5), (2, 0), (4, 2)\},&
 \{(0, 7), (4, 0), (4, 2)\},&
\{(0, 7), (3, 0), (7, 3)\}.

\end{array}$$}

\noindent10.~$(n,m)=(15,17)$
{\scriptsize $$\begin{array}{llll}
\{(0, 0), (0, 2), (7, 8)\},&
\{(0, 0), (0, 1), (6, 7)\},&
\{(0, 0), (0, 4), (5, 8)\},&
\{(0, 0), (0, 5), (4, 8)\},\\
\{(0, 0), (0, 3), (7, 5)\},&
\{(0, 0), (0, 6), (5, 7)\},&
\{(0, 0), (0, 7), (6, 5)\},&
\{(0, 0), (0, 8), (7, 4)\},\\
\{(0, 0), (1, 0), (3, 8)\},&
\{(0, 0), (1, 2), (4, 7)\},&
\{(0, 0), (1, 1), (7, 3)\},&
\{(0, 0), (1, 4), (5, 6)\},\\
\{(0, 0), (1, 5), (6, 4)\},&
\{(0, 0), (1, 6), (3, 7)\},&
\{(0, 0), (1, 7), (7, 1)\},&
\{(0, 0), (1, 8), (4, 6)\},\\
\{(0, 0), (2, 4), (6, 3)\},&
\{(0, 0), (2, 0), (7, 0)\},&
\{(0, 1), (2, 6), (7, 0)\},&
\{(0, 1), (2, 8), (3, 0)\},\\
\{(0, 3), (3, 7), (5, 0)\},&
\{(0, 2), (6, 3), (7, 0)\},&
\{(0, 1), (5, 4), (6, 0)\},&
\{(0, 7), (3, 0), (7, 4)\},\\
\{(0, 7), (1, 0), (4, 0)\},&
\{(0, 8), (2, 2), (4, 0)\},&
\{(0, 6), (4, 0), (6, 6)\},&
\{(0, 4), (5, 6), (6, 0)\},\\
\{(0, 7), (3, 8), (5, 0)\},&
\{(0, 8), (3, 0), (7, 1)\},&
\{(0, 6), (3, 0), (6, 3)\},&
\{(0, 7), (5, 5), (6, 0)\},\\
\{(0, 5), (4, 5), (6, 0)\},&
\{(0, 8), (3, 3), (5, 0)\},&
\{(0, 8), (4, 4), (6, 0)\},&
\{(0, 8), (4, 3), (7, 0)\},\\
\{(0, 5), (4, 2), (5, 0)\},&
\{(0, 5), (2, 4), (7, 0)\},&
\{(0, 6), (3, 2), (7, 0)\},&
\{(0,1),(1,0),(2,3),(4,6),(7,8)\}.
\end{array}$$}

\noindent11.~$(n,m)=(15,23)$
{\scriptsize $$\begin{array}{llll}
 \{(0, 0), (4, 0), (7, 11)\},&
\{(0, 0), (3, 1), (7, 10)\},&
\{(0, 0), (1, 10), (6, 11)\},&
\{(0, 0), (5, 10), (7, 8)\},\\
 \{(0, 0), (5, 11), (6, 5)\},&
\{(0, 0), (2, 11), (6, 10)\},&
\{(0, 0), (0, 5), (6, 9)\},&
\{(0, 0), (1, 5), (7, 7)\},\\
 \{(0, 0), (6, 8), (7, 1)\},&
\{(0, 0), (0, 9), (4, 11)\},&
\{(0, 0), (4, 10), (6, 3)\},&
\{(0, 0), (0, 6), (5, 9)\},\\
 \{(0, 0), (0, 7), (7, 6)\},&
\{(0, 0), (0, 11), (6, 6)\},&
\{(0, 0), (1, 8), (7, 5)\},&
\{(0, 0), (1, 11), (7, 4)\},\\
 \{(0, 0), (5, 7), (7, 2)\},&
\{(0, 0), (3, 10), (7, 3)\},&
\{(0, 0), (0, 10), (5, 6)\},&
\{(0, 0), (1, 0), (4, 8)\},\\
 \{(0, 0), (3, 9), (5, 5)\},&
\{(0, 0), (2, 10), (4, 7)\},&
\{(0, 0), (3, 5), (7, 0)\},&
\{(0, 10), (2, 0), (7, 4)\},\\
 \{(0, 11), (5, 0), (7, 9)\},&
\{(0, 11), (4, 0), (7, 7)\},&
\{(0, 11), (1, 0), (1, 2)\},&
\{(0, 5), (1, 0), (5, 4)\},\\
 \{(0, 10), (3, 0), (5, 8)\},&
\{(0, 0), (1, 9), (6, 0)\},&
\{(0, 0), (1, 3), (4, 5)\},&
\{(0, 11), (3, 0), (4, 1)\},\\
 \{(0, 10), (1, 0), (6, 0)\},&
\{(0, 11), (2, 0), (5, 3)\},&
\{(0, 8), (2, 0), (4, 6)\},&
\{(0, 0), (1, 7), (4, 1)\},\\
 \{(0, 11), (1, 7), (6, 0)\},&
\{(0, 0), (5, 2), (6, 1)\},&
\{(0, 10), (5, 0), (7, 5)\},&
\{(0, 8), (1, 0), (3, 7)\},\\
 \{(0, 6), (5, 0), (6, 4)\},&
\{(0, 9), (2, 0), (5, 4)\},&
\{(0, 9), (6, 0), (6, 8)\},&
\{(0, 3), (5, 0), (7, 0)\},\\
 \{(0, 9), (3, 0), (4, 6)\},&
\{(0, 8), (2, 2), (3, 0)\},&
\{(0, 0), (0, 4), (2, 3)\},&
\{(0, 4), (3, 0), (6, 0)\},\\
 \{(0, 11), (3, 4), (7, 0)\},&
\{(0, 8), (3, 3), (6, 0)\},&
\{(0, 9), (1, 6), (7, 0)\},&
\{(0, 1),(1, 3),(3, 7),(5, 9),(7, 10)\},\\
\{(0, 8), (7, 0), (7, 1)\},&
 \{(0, 10), (3, 8), (7, 0)\},&
\{(0, 9), (4, 0), (4, 3)\}.

\end{array}$$}

\section{ $(n\times m,\{3,6\},1)$-GDFs with small parameters}\label{app-C}
The parameters and the base blocks   are listed below.\\

\noindent1.~$(n,m)=(5,11)$

{\scriptsize $$\begin{array}{llll}
\{(0, 0), (1, 1), (2, 5)\},&
\{(0, 0), (1, 5), (2, 1)\},&
\{(0, 3), (1, 5), (2, 0)\},&
\{(0, 2), (0, 5), (1, 2), (2, 0), (2,1),(2,5)\},\\

\{(0, 0), (0, 2), (2, 4)\}.

\end{array}$$}

\noindent2.~$(n,m)=(5,17)$
{ \scriptsize$$\begin{array}{llll}
\{(1, 0), (2, 0), (0, 8)\},&
\{(0, 0), (2, 0), (0, 7)\},&
\{(2, 0), (0, 1), (0, 6)\},&
\{(1, 0), (0, 1), (2, 8),(0, 4),(0,5),(2,6)\},\\

\{(0, 0), (1, 3), (2, 8)\},&
\{(0, 2), (2, 0), (2, 8)\},&
\{(1, 0), (2, 2), (0, 6)\},&
\{(2, 0), (1, 2), (0, 5)\},\\

\{(2, 0), (1, 7), (0, 3)\},&
\{(0, 0), (1, 1), (1, 7)\}.

\end{array}$$}

\noindent3.~$(n,m)=(5,23)$

{  \scriptsize$$\begin{array}{llll}
\{(1, 0), (2, 0), (0, 11)\},&
\{(0, 0), (2, 0), (0, 10)\},&
\{(2, 0), (0, 9), (0, 1)\},&
\{(2, 0), (1, 8), (0, 2),(2, 9),(2,11),(0,6)\},\\

\{(0, 0), (0, 1), (2, 11)\},&
\{(2, 0), (0, 7), (1, 1)\},&
\{(0, 0), (0, 6), (2, 1)\},&
\{(2, 0), (0, 3), (0, 8)\},\\

\{(1, 0), (0, 10), (0, 3)\},&
\{(0, 0), (1, 9), (2, 2)\},&
\{(0, 5), (1, 0), (2, 11)\},&
\{(2, 0), (0, 4), (1, 9)\},\\
\{(0, 0), (1, 7), (0, 3)\},&
\{(1, 0), (0, 2), (2, 10)\},&
\{(1, 0), (2, 8), (0, 4)\}.

\end{array}$$}

\noindent4.~$(n,m)=(5,29)$
{ \scriptsize$$\begin{array}{llll}

\{(0, 13), (2, 0), (2, 14)\},& \{(0, 14), (1, 0), (2, 7)\},& \{(0, 11), (1, 0), (2, 9)\},&
\{(0,0),(0,4),(0,7),(1,8),(2,11),(2,13)\},\\

 \{(0, 0), (1, 12), (2, 12)\},& \{(0, 11), (2, 0), (2, 6)\},& \{(0, 10), (2, 0), (2, 13)\},& \{(0, 5), (1, 0), (2, 10)\},\\
 \{(0, 6), (1, 0), (2, 14)\},& \{(0, 3), (1, 0), (2, 13)\},& \{(0, 14), (1, 1), (2, 0)\},& \{(0, 8), (2, 0), (2, 8)\},\\
\{(0, 9), (1, 0), (1, 5)\},& \{(0, 12), (1, 0), (2, 11)\},& \{(0, 9), (1, 2), (2, 0)\},& \{(0, 0), (0, 12), (2, 14)\},\\
\{(0, 0), (0, 10), (0, 11)\},& \{(0, 8), (1, 0), (2, 2)\},& \{(0, 12), (2, 0), (2, 9)\},& \{(0, 10), (1, 0), (2, 6)\}.
\end{array}$$}

\noindent5.~$(n,m)=(5,35)$
{ \scriptsize$$\begin{array}{llll}

\{(0, 10), (2, 0), (2, 15)\},& \{(0, 14), (1, 4), (2, 0)\},& \{(0, 12), (2, 0), (2, 10)\},&
\{(0,0),(0,4),(0,7),(1,8),(2,11),(2,13)\},\\
 \{(0, 13), (2, 0), (2, 16)\},& \{(0, 17), (1, 12), (2, 0)\},& \{(0, 17), (1, 0), (2, 11)\},& \{(0, 11), (2, 0), (2, 12)\},\\
\{(0, 15), (1, 0), (2, 14)\},& \{(0, 0), (1, 2), (1, 13)\},& \{(0, 15), (1, 1), (2, 0)\},& \{(0, 9), (1, 3), (2, 0)\},\\
 \{(0, 8), (1, 0), (1, 1)\},& \{(0, 16), (1, 0), (2, 12)\},& \{(0, 7), (2, 0), (2, 17)\},& \{(0, 9), (1, 0), (2, 17)\},\\
 \{(0, 0), (2, 12), (2, 17)\},& \{(0, 11), (1, 0), (2, 6)\},& \{(0, 0), (0, 13), (2, 15)\},& \{(0, 2), (1, 0), (1, 9)\},\\
 \{(0, 0), (1, 16), (2, 16)\},& \{(0, 13), (1, 0), (2, 10)\},& \{(0, 16), (2, 0), (2, 8)\},& \{(0, 0), (0, 6), (1, 15)\},\\
 \{(0, 0), (2, 0), (2, 14)\}.

\end{array}$$}

\noindent6.~$(n,m)=(7,13)$

{ \scriptsize$$\begin{array}{llll}
\{(0, 6), (3, 0), (2, 0)\},&
\{(0, 0), (2, 0), (3, 6)\},&
\{(0, 0), (3, 0), (3, 5)\},&
\{(2, 0), (0, 5), (2, 1),(1, 6),(0, 2),(3, 3)\},\\

\{(1, 0), (3, 6), (0, 2)\},&
\{(1, 0), (0, 1), (3, 4)\},&
\{(0, 0), (3, 2), (0, 6)\},&
\{(0, 5), (3, 0), (0, 1)\},\\

\{(0, 0), (2, 5), (0, 2)\},&
\{(2, 2), (0, 0), (1, 5)\},&
\{(1, 0), (3, 1), (0, 4)\}.
\end{array}$$}

\noindent7.~$(n,m)=(7,19)$

{\scriptsize$$\begin{array}{llll}
\{(2, 0), (3, 0), (0, 9)\},&
\{(0, 0), (2, 0), (3, 9)\},&
\{(0, 0), (3, 8), (3, 0)\},&
\{(0, 0), (0, 5), (3, 1),(2,7),(3,3),(0,9)\},\\

\{(2, 0), (3, 8), (0, 1)\},&
\{(0, 0), (0, 1), (2, 9)\},&
\{(3, 0), (3, 6), (0, 1)\},&
\{(0, 0), (1, 1), (3, 6)\},\\

\{(0, 0), (1, 7), (3, 2)\},&
\{(0, 0), (2, 1), (3, 4)\},&
\{(0, 0), (1, 2), (2, 6)\},&
\{(2, 0), (1, 5), (0, 8)\},\\

\{(2, 0), (0, 3), (1, 9)\},&
\{(1, 6), (3, 0), (0, 7)\},&
\{(2, 0), (0, 4), (0, 7)\},&
\{(0, 5), (3, 0), (2, 8)\},\\
\{(2, 7), (3, 0), (0, 3)\},&
\{(0, 0), (0, 7), (1, 5)\}.
\end{array}$$}

\noindent8.~$(n,m)=(7,37)$
{ \scriptsize$$\begin{array}{llll}

\{(0, 12), (1, 3), (3, 0)\},& \{(0, 18), (1, 0), (3, 3)\},& \{(0, 5), (1, 0), (3, 13)\},&
\{(0,0),(0,5),(0,9),(2,7),(3,1),(3,3)\},\\

 \{(0, 0), (1, 5), (3, 17)\},& \{(0, 13), (1, 13), (3, 0)\},& \{(0, 18), (3, 0), (3, 17)\},& \{(0, 0), (0, 15), (2, 16)\},\\
\{(0, 14), (2, 0), (3, 9)\},& \{(0, 10), (3, 0), (3, 16)\},& \{(0, 12), (2, 0), (2, 7)\},& \{(0, 16), (1, 1), (3, 0)\},\\
 \{(0, 0), (1, 13), (2, 17)\},& \{(0, 14), (1, 0), (3, 11)\},& \{(0, 14), (2, 7), (3, 0)\},& \{(0, 0), (0, 10), (2, 15)\},\\
 \{(0, 0), (0, 3), (3, 15)\},& \{(0, 0), (1, 3), (1, 16)\},& \{(0, 0), (0, 1), (3, 14)\},& \{(0, 18), (2, 0), (3, 18)\},\\
\{(0, 3), (1, 0), (3, 14)\},& \{(0, 17), (2, 1), (3, 0)\},& \{(0, 4), (2, 0), (2, 12)\},& \{(0, 12), (1, 0), (2, 2)\},\\
 \{(0, 0), (1, 1), (1, 7)\},& \{(0, 17), (1, 0), (2, 8)\},& \{(0, 0), (1, 12), (2, 18)\},& \{(0, 9), (3, 0), (3, 14)\},\\
 \{(0, 11), (3, 0), (3, 18)\},& \{(0, 0), (0, 8), (3, 18)\},& \{(0, 8), (2, 0), (3, 17)\},& \{(0, 11), (2, 0), (3, 15)\},\\
 \{(0, 16), (1, 0), (3, 9)\},& \{(0, 17), (2, 0), (2, 11)\},& \{(0, 10), (1, 0), (2, 14)\},& \{(0, 0), (2, 6), (3, 16)\},\\
 \{(0, 15), (1, 2), (2, 0)\},& \{(0, 11), (1, 0), (2, 11)\},& \{(0, 8), (1, 0), (3, 10)\}.
\end{array}$$}

\noindent9.~$(n,m)=(11,11)$

{ \scriptsize$$\begin{array}{llll}
\{(4, 0), (5, 0), (0, 5)\},&
\{(0, 0), (2, 0), (5, 5)\},&
\{(2, 0), (0, 4), (5, 0)\},&
\{(0, 0), (0, 3), (0, 4),(3,1),(2,2),(4,5)\},\\

\{(0, 0), (2, 5), (4, 0)\},&
\{(5, 0), (2, 5), (0, 1)\},&
\{(4, 0), (1, 4), (0, 1)\},&
\{(1, 1), (4, 0), (0, 3)\},\\

\{(5, 0), (0, 2), (5, 5)\},&
\{(0, 4), (4, 0), (5, 1)\},&
\{(4, 0), (0, 2), (3, 5)\},&
\{(0, 0), (3, 4), (5, 1)\},\\
\{(0, 0), (0, 2), (5, 4)\},&
\{(0, 0), (4, 3), (5, 0)\},&
\{(0, 0), (1, 2), (4, 4)\},&
\{(0, 0), (1, 5), (2, 1)\}.

\end{array}$$}

\noindent10.~$(n,m)=(13,13)$
{ \scriptsize$$\begin{array}{llll}
\{(0, 4), (3, 0), (5, 0)\},&
\{(0, 2), (3, 0), (6, 5)\},&
 \{(0, 1), (1, 0), (6, 3)\},&
\{(0,0),(1,1),(3,3),(4,5),(5,2),(6,0)\},\\

\{(0, 5), (4, 0), (5, 5)\},&
 \{(0, 0), (0, 3), (1, 3)\},&
 \{(0, 2), (2, 0), (6, 6)\},&
\{(0, 5), (3, 0), (6, 2)\},\\
 \{(0, 6), (6, 0), (6, 5)\},&
 \{(0, 6), (1, 0), (5, 4)\},&
 \{(0, 6), (4, 0), (6, 4)\},&
 \{(0, 0), (2, 6), (6, 6)\},\\
 \{(0, 3), (5, 0), (6, 4)\},&
 \{(0, 0), (0, 2), (5, 6)\},&
 \{(0, 6), (2, 0), (3, 6)\},&
\{(0, 5), (6, 0), (6, 1)\},\\
\{(0, 0), (0, 4), (5, 5)\},&
 \{(0, 1), (3, 0), (6, 6)\},&
 \{(0, 5), (1, 0), (4, 1)\},&
\{(0, 3), (2, 0), (4, 5)\},\\
 \{(0, 3), (4, 0), (4, 6)\},&
 \{(0, 6), (3, 0), (5, 1)\},&
\{(0, 4), (2, 0), (4, 3)\},&
 \{(0, 6), (1, 2), (5, 0)\}.

\end{array}$$}

\noindent11.~$(n,m)=(13,19)$
{ \scriptsize$$\begin{array}{llll}
\{(0, 7), (4, 0), (6, 5)\},&
 \{(0, 7), (5, 0), (5, 7)\},&
 \{(0, 9), (2, 2), (6, 0)\},&
\{(0,0),(1,1),(3,3),(4,5),(5,2),(6,0)\},\\

 \{(0, 0), (0, 1), (6, 3)\},&
\{(0, 0), (2, 9), (6, 9)\},&
 \{(0, 1), (3, 0), (6, 9)\},&
 \{(0, 6), (5, 0), (6, 7)\},\\
 \{(0, 9), (4, 8), (5, 0)\},&
 \{(0, 0), (0, 6), (5, 7)\},&
 \{(0, 5), (5, 0), (6, 9)\},&
\{(0, 7), (4, 2), (6, 0)\},\\
 \{(0, 0), (0, 2), (4, 8)\},&
 \{(0, 6), (4, 0), (4, 9)\},&
 \{(0, 8), (4, 0), (6, 4)\},&
 \{(0, 9), (4, 0), (5, 6)\},\\
 \{(0, 2), (3, 0), (6, 8)\},&
 \{(0, 3), (2, 0), (3, 8)\},&
 \{(0, 8), (2, 0), (3, 3)\},&
 \{(0, 8), (2, 4), (5, 0)\},\\
 \{(0, 0), (3, 0), (6, 7)\},&
 \{(0, 0), (0, 4), (5, 8)\},&
 \{(0, 8), (3, 0), (4, 4)\},&
\{(0, 0), (0, 3), (2, 6)\},\\
 \{(0, 3), (4, 0), (6, 8)\},&
 \{(0, 0), (1, 0), (5, 9)\},&
 \{(0, 9), (1, 0), (6, 3)\},&
 \{(0, 9), (2, 0), (6, 4)\},\\
 \{(0, 5), (1, 0), (2, 5)\},&
 \{(0, 1), (1, 0), (5, 7)\},&
 \{(0, 9), (3, 0), (5, 7)\},&
 \{(0, 8), (6, 0), (6, 5)\},\\
 \{(0, 7), (3, 0), (3, 8)\},&
 \{(0, 7), (1, 0), (3, 1)\},&
 \{(0, 4), (1, 0), (4, 6)\},&
 \{(0, 6), (2, 0), (5, 2)\},\\

\{(0, 6), (1, 0), (6, 5)\}.
\end{array}$$}

\noindent12.~$(n,m)=(13,37)$
{ \scriptsize$$\begin{array}{llll}
\{(0, 5), (4, 0), (6, 12)\},& \{(0, 17), (2, 0), (6, 12)\},& \{(0, 14), (2, 0), (6, 15)\},&
 \{(0,0),(1,1),(3,3),(4,5),(5,2),(6,0)\},\\
\{(0, 11), (4, 9), (6, 0)\},& \{(0, 0), (4, 6), (4, 8)\},& \{(0, 13), (3, 2), (6, 0)\},& \{(0, 2), (5, 0), (6, 8)\},\\
 \{(0, 3), (2, 0), (2, 17)\},& \{(0, 0), (1, 12), (5, 14)\},& \{(0, 18), (2, 0), (6, 4)\},& \{(0, 15), (1, 0), (6, 18)\},\\
\{(0, 18), (1, 0), (6, 6)\},& \{(0, 12), (3, 0), (5, 5)\},& \{(0, 18), (5, 0), (6, 14)\},& \{(0, 16), (4, 0), (4, 15)\},\\
 \{(0, 0), (1, 15), (6, 16)\},& \{(0, 9), (3, 8), (4, 0)\},& \{(0, 17), (4, 5), (5, 0)\},& \{(0, 17), (3, 0), (6, 7)\},\\
\{(0, 16), (1, 0), (5, 7)\},& \{(0, 0), (0, 11), (3, 12)\},& \{(0, 7), (6, 0), (6, 16)\},& \{(0, 15), (3, 10), (6, 0)\},\\
\{(0, 14), (3, 0), (5, 0)\},& \{(0, 12), (4, 9), (5, 0)\},& \{(0, 7), (2, 0), (6, 17)\},& \{(0, 15), (4, 0), (4, 4)\},\\
 \{(0, 8), (5, 0), (5, 3)\},& \{(0, 16), (6, 0), (6, 10)\},& \{(0, 15), (2, 0), (5, 18)\},& \{(0, 16), (3, 0), (3, 12)\},\\
\end{array}$$}

\noindent13.~$(n,m)=(19,19)$
{ \scriptsize$$\begin{array}{llll}
\{(0, 4), (5, 0), (7, 9)\},& \{(0, 0), (0, 7), (7, 8)\},& \{(0, 9), (6, 0), (7, 7)\},&
\{(0,0),(1,1),(3,3),(4,5),(5,2),(6,0)\},\\
\{(0, 0), (7, 0), (9, 4)\},& \{(0, 9), (3, 0), (8, 4)\},& \{(0, 9), (2, 1), (5, 0)\},& \{(0, 7), (7, 0), (8, 5)\},\\ \{(0, 0), (2, 6), (6, 8)\},& \{(0, 6), (3, 0), (6, 7)\},& \{(0, 5), (3, 1), (9, 0)\},& \{(0, 8), (4, 0), (7, 2)\},\\ \{(0, 9), (9, 0), (9, 3)\},& \{(0, 2), (5, 0), (7, 8)\},& \{(0, 2), (6, 0), (6, 9)\},& \{(0, 7), (4, 0), (7, 9)\},\\

\{(0, 8), (6, 0), (9, 5)\},& \{(0, 0), (3, 1), (9, 5)\},& \{(0, 9), (7, 0), (7, 6)\},& \{(0, 0), (3, 0), (8, 5)\},\\

\{(0, 0), (5, 0), (8, 6)\},& \{(0, 0), (3, 8), (8, 9)\},& \{(0, 8), (3, 6), (7, 0)\},& \{(0, 0), (4, 3), (9, 6)\},\\ \{(0, 7), (2, 0), (8, 9)\},& \{(0, 6), (1, 0), (8, 9)\},& \{(0, 5), (3, 0), (7, 8)\},& \{(0, 1), (4, 0), (8, 9)\},\\ \{(0, 7), (6, 0), (8, 1)\},& \{(0, 9), (4, 0), (8, 6)\},& \{(0, 4), (1, 4), (9, 0)\},& \{(0, 0), (6, 3), (6, 5)\},\\
\{(0, 0), (1, 8), (9, 8)\},& \{(0, 6), (5, 0), (5, 1)\},& \{(0, 7), (3, 0), (5, 0)\},& \{(0, 4), (7, 0), (9, 3)\},\\

 \{(0, 0), (5, 9), (9, 9)\},& \{(0, 9), (2, 0), (6, 4)\},& \{(0, 6), (2, 2), (6, 0)\},& \{(0, 5), (1, 4), (7, 0)\},\\ \{(0, 7), (8, 0), (8, 8)\},& \{(0, 8), (3, 0), (7, 7)\},& \{(0, 7), (5, 4), (9, 0)\},& \{(0, 8), (8, 0), (9, 6)\},\\ \{(0, 0), (9, 2), (9, 7)\},& \{(0, 8), (1, 1), (9, 0)\},& \{(0, 6), (2, 0), (9, 7)\},& \{(0, 9), (1, 0), (6, 6)\},\\ \{(0, 3), (2, 0), (7, 7)\},& \{(0, 9), (8, 0), (9, 9)\},& \{(0, 0), (1, 4), (2, 7)\},& \{(0, 8), (1, 0), (1, 4)\},\\ \{(0, 8), (4, 5), (5, 0)\},& \{(0, 0), (2, 5), (8, 7)\},& \{(0, 2), (2, 0), (8, 6)\},& \{(0, 5), (4, 0), (9, 8)\}.

\end{array}$$}

\noindent14.~$(n,m)=(19,37)$
{ \scriptsize$$\begin{array}{llll}
\{(0, 17), (4, 0), (8, 16)\},& \{(0, 0), (3, 9), (3, 18)\},& \{(0, 2), (7, 0), (7, 16)\},&
\{(0,0),(1,1),(3,3),(4,5),(5,2),(6,0)\},\\

 \{(0, 14), (1, 0), (6, 6)\},& \{(0, 12), (8, 0), (8, 18)\},& \{(0, 11), (4, 0), (5, 4)\},& \{(0, 16), (4, 0), (9, 18)\},\\
 \{(0, 15), (5, 0), (6, 11)\},& \{(0, 15), (9, 0), (9, 8)\},& \{(0, 0), (7, 13), (7, 18)\},& \{(0, 0), (1, 10), (9, 12)\},\\
\{(0, 15), (3, 0), (8, 0)\},& \{(0, 0), (2, 13), (5, 14)\},& \{(0, 14), (2, 0), (7, 9)\},& \{(0, 13), (4, 0), (5, 8)\},\\
 \{(0, 13), (3, 0), (8, 11)\},& \{(0, 14), (3, 6), (9, 0)\},& \{(0, 9), (2, 0), (4, 7)\},& \{(0, 0), (8, 5), (8, 12)\},\\
 \{(0, 4), (2, 0), (6, 12)\},& \{(0, 17), (3, 0), (3, 10)\},& \{(0, 0), (1, 14), (8, 15)\},& \{(0, 3), (2, 0), (3, 17)\},\\
 \{(0, 15), (2, 0), (8, 4)\},& \{(0, 17), (8, 0), (9, 16)\},& \{(0, 3), (8, 0), (8, 11)\},& \{(0, 1), (3, 0), (4, 12)\},\\
 \{(0, 13), (7, 0), (9, 11)\},& \{(0, 16), (5, 0), (5, 6)\},& \{(0, 6), (4, 0), (7, 17)\},& \{(0, 0), (6, 3), (9, 16)\},\\
\end{array}$$}

\noindent15.~$(n,m)=(37,37)$
{ \scriptsize$$\begin{array}{llll}
\{(0, 11), (6, 0), (13, 1)\},& \{(0, 0), (6, 1), (12, 11)\},& \{(0, 3), (7, 0), (16, 11)\},&
 \{(0,0),(1,1),(3,3),(4,5),(5,2),(6,0)\},\\

\{(0, 15), (18, 0), (18, 14)\},& \{(0, 17), (6, 0), (11, 9)\},& \{(0, 16), (1, 0), (5, 8)\},& \{(0, 15), (4, 2), (13, 0)\},\\
 \{(0, 17), (10, 0), (18, 13)\},& \{(0, 10), (7, 0), (14, 4)\},& \{(0, 16), (15, 0), (15, 15)\},& \{(0, 18), (8, 0), (18, 0)\},\\
 \{(0, 0), (9, 9), (11, 12)\},& \{(0, 14), (6, 0), (18, 4)\},& \{(0, 0), (3, 8), (4, 13)\},& \{(0, 0), (10, 1), (17, 18)\},\\
\{(0, 17), (12, 0), (17, 5)\},& \{(0, 17), (2, 0), (18, 10)\},& \{(0, 13), (1, 4), (13, 0)\},& \{(0, 7), (15, 0), (17, 18)\},\\
\{(0, 17), (14, 0), (17, 15)\},& \{(0, 18), (4, 11), (14, 0)\},& \{(0, 8), (4, 0), (12, 2)\},& \{(0, 11), (5, 0), (11, 12)\},\\
\{(0, 0), (9, 0), (14, 3)\},& \{(0, 18), (7, 0), (18, 5)\},& \{(0, 2), (7, 0), (15, 17)\},& \{(0, 12), (1, 0), (5, 7)\},\\
 \{(0, 11), (7, 0), (18, 16)\},& \{(0, 5), (16, 0), (16, 4)\},& \{(0, 6), (6, 0), (17, 3)\},& \{(0, 17), (6, 7), (11, 0)\},\\
\end{array}$$}

\newpage
\section{$(n,3,1)$-GDPs and $(n,\{3,5\},1)$-GDPs with leave $[h]^r$ \\  and  small parameters}\label{app-D}
The parameters and the base blocks   are listed below.\\\

\noindent1.
$(n,h,r)=(17,5,2)$\vspace{0.22cm}
{\scriptsize$$\begin{array}{llllll}
\hspace{-7.3cm}\{0,1,7\},&\{0,3,8\}.
\end{array}$$}

\noindent2.
$(n,h,r)=(23,5,2)$\vspace{0.22cm}
{\scriptsize$$\begin{array}{llllll}
\hspace{-7.3cm}\{0,1,9\},&\{0,3,10\},&\{0,5,11\}.
\end{array}$$}

\noindent3.
$(n,h,r)=(25,7,3)$\vspace{0.22cm}
{\scriptsize$$\begin{array}{llllll}
\hspace{-7.3cm}\{0,2,7\},&\{0,8,12\},&\{0,10,11\}.
\end{array}$$}

\noindent4.
$(n,h,r)=(29,5,3)$\vspace{0.22cm}
{\scriptsize$$\begin{array}{llllll}
\hspace{-7.3cm}\{0,1,10\}, &\{0,2,14\},&\{0,4,11\},&\{0,5,13\}.
\end{array}$$}

\noindent5.
$(n,h,r)=(31,7,3)$\vspace{0.22cm}
{\scriptsize$$\begin{array}{llllll}
\hspace{-7.3cm}\{0,1,15\},&\{0,4,11\},&\{0,5,13\},&\{0,10,12\}.
\end{array}$$}

\noindent6.
$(n,h,r)=(33,9,3)$\vspace{0.22cm}
{\scriptsize$$\begin{array}{llllll}
\hspace{-7.3cm}\{0,1,8\}, &\{0,5,16\}, &\{0,10,14\}, &\{0,13,15\}.
\end{array}$$}

\noindent7.
$(n,h,r)=(35,5,3)$\vspace{0.22cm}
{\scriptsize$$\begin{array}{llllll}
\hspace{-6.9cm}\{0,1,12\}, &\{0,2,17\},&\{0,4,14\}, &\{0,5,13\},&\{0,7,16\}.
\end{array}$$}

\noindent8.
$(n,h,r)=(39,9,3)$\vspace{0.22cm}
{\scriptsize$$\begin{array}{llllll}
\hspace{-6.9cm}\{0,1,11\}, &\{0,2,18\}, &\{0,5,19\}, &\{0,13,17\}, &\{0,7,15\}.
\end{array}$$}

\noindent9.
$(n,h,r)=(41,11,2)$\vspace{0.22cm}
{\scriptsize$$\begin{array}{llllll}
\hspace{-6.9cm}\{0,1,16\}, &\{0,3,17\},&\{0,5,18\}, &\{0,7,19\},&\{0,9,20\}.
\end{array}$$}

\noindent10.
$(n,h,r)=(41,9,3)$\vspace{0.22cm}
{\scriptsize$$\begin{array}{llllll}
\hspace{-7.4cm}\{0,8, 19\},&\{0, 7, 17\},&\{0, 2, 15, 16, 20\}.
\end{array}$$}

\noindent11.
$(n,h,r)=(43,13,3)$\vspace{0.22cm}
{\scriptsize$$\begin{array}{llllll}
\hspace{-6.4cm}\{0,1,20\}, &\{0,2,10\},&\{0,4,17\}, &\{0,14,21\},&\{0,11,16\}.
\end{array}$$}

\noindent12.
$(n,h,r)=(47,11,2)$\vspace{0.22cm}
{\scriptsize$$\begin{array}{lllllll}
\hspace{-4.8cm}\{0,1,18\}, &\{0,3,19\},&\{0,5,20\},&\{0,7,21\},&\{0,9,22\},&\{0,11,23\}.
\end{array}$$}

\noindent13.
$(n,h,r)=(47,9,3)$\vspace{0.22cm}
{\scriptsize$$\begin{array}{llllllll}
\hspace{-7.3cm}\{0, 10, 23\},&\{ 0, 8,19\},&\{0, 7, 21\},&\{0, 2, 17, 18, 22\}.
\end{array}$$}

\noindent14.
$(n,h,r)=(53,11,3)$\vspace{0.22cm}
{\scriptsize$$\begin{array}{llllllll}
\hspace{-3.4cm}\{0,1,26\}, &\{0,4,23\}, &\{0,2,20\}, &\{0,8,22\}, &\{0,13,24\}, & \{0,5,21\},&\{0,7,17\}.
\end{array}$$}

\noindent15.
$(n,h,r)=(57,9,3)$\vspace{0.22cm}
{\scriptsize$$\begin{array}{llllllllll}
\hspace{-1.3cm}\{0,2,26\}, &\{0,4,20\}, &\{0,5,28\}, &\{0,11,21\}, &\{0,8,25\}, &\{0,13,27\},&\{0,15,22\}, &\{0,18,19\}.
\end{array}$$}

\noindent16.
$(n,h,r)=(59,11,3)$\vspace{0.22cm}
{\scriptsize$$\begin{array}{llllllll}
\hspace{-1.5cm}\{0,1,14\}, &\{0,2,22\}, &\{0,4,27\}, &\{0,8,24\}, &\{0,10,29\}, & \{0,11,28\},&\{0,7,25\}, &\{0,5,26\}.
\end{array}%
$$}

\section{ $(n\times m,\{3,4,5\},1)$-GDFs with small parameters}\label{app-E}
The parameters and the  base blocks   are listed below.\\

\noindent1.~$(n,m)=(5,45)$

{\scriptsize $$%
$$}

\newpage
\section{ Program}\label{app-F}
The following is the program regarding exhaustive search in Lemma \ref{non345}.

\begin{lstlisting}[language=Python]
from itertools import product
import sys

groups = [[0, 1, 2], [0, 2, 2], [0, 1, 2, 2], [0, 0, 1, 2, 2], [0, 0, 1, 2, 2]]
target_to_bit = {}
bit = 0
for v in range(1, 7):
    target_to_bit[(0, v)] = bit
    bit += 1
for dy in range(-6, 7):
    target_to_bit[(1, dy)] = bit
    bit += 1
for dy in range(-6, 7):
    target_to_bit[(2, dy)] = bit
    bit += 1

ALL_MASK = (1 << 32) - 1

def get_valid_configs(xs):
    n = len(xs)
    configs = []
    for ys in product(range(7), repeat=n):
        used_bits = 0
        ok = True
        for i in range(n):
            for j in range(i + 1, n):
                dx = xs[j] - xs[i]
                dy = ys[j] - ys[i]

                if dx == 0:
                    abs_dy = abs(dy)
                    if abs_dy not in range(1, 7):
                        ok = False
                        break
                    key = (0, abs_dy)
                elif dx == 1:
                    if dy not in range(-6, 7):
                        ok = False
                        break
                    key = (1, dy)
                elif dx == 2:
                    if dy not in range(-6, 7):
                        ok = False
                        break
                    key = (2, dy)
                else:
                    ok = False
                    break

                b = target_to_bit[key]

                if used_bits & (1 << b):
                    ok = False
                    break
                used_bits |= (1 << b)

            if not ok:
                break
        if ok:
            configs.append((ys, used_bits))
    return configs


print("Precomputing valid configurations for each group...")
all_configs = []
for i, xs in enumerate(groups):
    cfgs = get_valid_configs(xs)
    print(f"  Group {i+1} (x={xs}): {len(cfgs)} valid configurations")
    all_configs.append(cfgs)
    if len(cfgs) == 0:
        print(f"Group {i+1} has no valid configuration, no solution.")
        sys.exit()


g4_configs = all_configs[3]
g5_configs = all_configs[4]

dict4 = {}
for ys, mask in g4_configs:
    if mask not in dict4:
        dict4[mask] = ys

dict5 = {}
for ys, mask in g5_configs:
    if mask not in dict5:
        dict5[mask] = ys

print(f"\nGroup 4 after deduplication has {len(dict4)} distinct masks")
print(f"Group 5 after deduplication has {len(dict5)} distinct masks")
total_pairs_last_two = (
    len(groups[3]) * (len(groups[3]) - 1) // 2
    + len(groups[4]) * (len(groups[4]) - 1) // 2)
pair_dict = {}
for mask4, ys4 in dict4.items():
    for mask5, ys5 in dict5.items():
        if mask4 & mask5:
            continue
        combined = mask4 | mask5
        if combined.bit_count() == total_pairs_last_two:
            if combined not in pair_dict:
                pair_dict[combined] = (ys4, ys5)
print(f"Group 4+5 valid pairings: {len(pair_dict)}")

c1 = all_configs[0]
c2 = all_configs[1]
c3 = all_configs[2]

print("\nStarting search for the first three groups...")
solution = [None] * 5

for ys1, mask1 in c1:
    for ys2, mask2 in c2:
        if mask1 & mask2:
            continue
        mask12 = mask1 | mask2
        for ys3, mask3 in c3:
            if mask12 & mask3:
                continue
            mask123 = mask12 | mask3
            remaining = ALL_MASK ^ mask123
            if remaining in pair_dict:
                ys4, ys5 = pair_dict[remaining]
                solution[0] = ys1
                solution[1] = ys2
                solution[2] = ys3
                solution[3] = ys4
                solution[4] = ys5
                print("\nSolution found!")
                for i, ys in enumerate(solution):
                    print(f"Group {i+1}: x={groups[i]}, y={list(ys)}")
                    print(f"   Points: {list(zip(groups[i], ys))}")
                print("\nAll ordered pairs :")
                all_pairs = []
                for gi, xs in enumerate(groups):
                    ys = solution[gi]
                    for i in range(len(xs)):
                        for j in range(i + 1, len(xs)):
                            dx = xs[j] - xs[i]
                            dy = ys[j] - ys[i]
                            if dx == 0:
                                all_pairs.append((0, abs(dy), gi + 1))
                            else:
                                all_pairs.append((dx, dy, gi + 1))
                for p in sorted(all_pairs):
                    print(p)
                unique = set((a, b) for a, b, _ in all_pairs)
                print(f"\nTotal {len(all_pairs)} pairs, {len(unique)} unique pairs")
                sys.exit()
print("\nNo solution found.")
print("This means that under the strict condition that each target pair appears exactly once, there is no second-coordinate assignment for these 5 groups that satisfies the condition.")
print("Therefore, this is a computational proof.")
\end{lstlisting}

\end{document}